\documentclass[letterpaper]{article}
\usepackage{uai2020}
\usepackage{smallref}
\usepackage[margin=1in]{geometry}
\usepackage{amsmath}
\usepackage{amssymb}
\usepackage{tikz}
\usetikzlibrary{arrows}
\usetikzlibrary{decorations.markings}
\usepackage{natbib}
\usepackage{paralist}

\usepackage{times}
\usepackage{bm}
\usepackage{bbm}
\usepackage{comment}
\usepackage{hyperref}
\usepackage[normalem]{ulem}

\newtheorem{definition}{Definition}
\newtheorem{proposition}{Proposition}
\newtheorem{lemma}{Lemma}
\newtheorem{theorem}{Theorem}
\newtheorem{corollary}{Corollary}
\newtheorem{assumption}{Assumption}
\newcommand{\qedsymbol}{$\square$}
\newcommand{\qed}{\hfill\qedsymbol}
\newenvironment{proof}[1][Proof.]{\noindent \textbf{#1\hskip.5em\relax}}{\qed\vskip\baselineskip}
\newenvironment{proof*}[1][Proof.]{\noindent \textbf{#1\hskip.5em\relax}}{}

\newcommand\C[1]{\mathcal{#1}}
\newcommand\B[1]{\bm{#1}}
\newcommand\BC[1]{\bm{\mathcal{#1}}}

\newcommand{\DMAG}{\mathrm{DMAG}}

\newcommand{\IM}{\mathrm{IM}}

\newcommand{\PAGFCI}{\C{P}_{\alg{FCI}}}
\newcommand{\PAGFCIJCI}{\C{P}_{\alg{FCI}_{\alg{JCI}}}}

\newcommand{\alg}[1]{\texttt{#1}}
\newcommand\acy{\mathrm{acy}}
\newcommand{\Prb}{\mathbb{P}}

\newcommand\given{\,|\,}
\DeclareMathOperator*{\CI}{{\,\perp\mkern-12mu\perp\,}}

\DeclareMathOperator*{\SEP}{\perp}
\DeclareMathOperator*{\nSEP}{\not\perp}
\newcommand\indep[4]{{#1} \CI_{#4} {#2} \given {#3}}

\newcommand{\dsep}[4]{{#1} \SEP_{#4}^d {#2} \given {#3}}

\newcommand{\sigmasep}[4]{{#1} \SEP_{#4}^\sigma {#2} \given {#3}}
\newcommand{\sigmacon}[4]{{#1} \nSEP_{#4}^\sigma {#2} \given {#3}}

\newcommand\mathbfsc[1]{\text{\normalfont\scshape#1}}

\newcommand\ansub[2]{\mathbfsc{an}_{#1}(#2)}
\newcommand\desub[2]{\mathbfsc{de}_{#1}(#2)}
\newcommand\sccsub[2]{\mathbfsc{sc}_{#1}(#2)}
\newcommand\pasub[2]{\mathbfsc{pa}_{#1}(#2)}

\newcommand\eref[1]{(\ref{#1})}

\newcommand{\ot}{\mathrel{\leftarrow}}
\newcommand{\oto}{\mathrel{\leftrightarrow}}
\newcommand{\ots}{\mathrel{{\leftarrow\mkern-11mu\ast}}}
\newcommand{\otc}{\mathrel{{\leftarrow\mkern-8mu\circ}}}
\newcommand{\sto}{\mathrel{{\ast\mkern-11mu\to}}}

\newcommand{\sts}{\mathrel{{\ast\mkern-11mu\relbar\mkern-9mu\relbar\mkern-11mu\ast}}}
\newcommand{\ctc}{\mathrel{{\circ\mkern-8mu\relbar\mkern-8mu\circ}}}
\newcommand{\cto}{\mathrel{{\circ\mkern-8mu\rightarrow}}}
\newcommand{\cts}{\mathrel{{\circ\mkern-8mu\relbar\mkern-9mu\relbar\mkern-11mu\ast}}}

\newcommand{\ctt}{\mathrel{{\circ\mkern-8mu\relbar\mkern-9mu\relbar}}}
\newcommand{\ttt}{\mathrel{{\relbar\mkern-9mu\relbar}}}
\newcommand{\ttc}{\mathrel{{\relbar\mkern-9mu\relbar\mkern-8mu\circ}}}

\newcommand{\typo}[2]{{\color{red}\sout{#1}}{\color{blue}{#2}}}
\newcommand{\typoadd}[1]{{\color{blue}{#1}}}
\newcommand{\typorm}[1]{{\color{red}\sout{#1}}}

\tikzstyle{var}=[circle,draw=black,fill=white,thick,minimum size=20pt,inner sep=0pt]
\tikzstyle{varh}=[circle,draw=gray,fill=white,thick,minimum size=20pt,inner sep=0pt,dashed]
\tikzstyle{arr}=[->,>=stealth',draw=black,thick]
\tikzstyle{arrh}=[->,>=stealth',draw=gray,fill=gray,thick,dashed]
\tikzstyle{biarr}=[<->,>=stealth',draw=black,fill=black,thick]
\tikzstyle{biarrh}=[<->,>=stealth',draw=gray,fill=gray,thick,dashed]
\tikzstyle{noarr}=[draw=black,fill=black,thick]
\tikzstyle{noarrh}=[draw=gray,fill=gray,thick,dashed]
\tikzstyle{fac}=[rectangle,draw=black!50,fill=black!20,thick,minimum size=10pt] 
\tikzstyle{varc}=[rectangle,draw=black,fill=white,thick,minimum size=20pt,inner sep=0pt]
\tikzstyle{varch}=[rectangle,draw=black,fill=white,thick,minimum size=20pt,inner sep=0pt,dashed]
\tikzstyle{carr}=[o->,>=stealth',draw=black,thick]
\tikzstyle{bicarr}=[o-o,>=stealth',draw=black,thick]
\tikzstyle{carc}=[o-o,>=stealth',draw=black,thick]

\title{Constraint-Based Causal Discovery using Partial Ancestral Graphs in the presence of Cycles}

\author{ {\bf Joris M.~Mooij} \\
Korteweg-de Vries Institute\\
University of Amsterdam\\
Amsterdam, The Netherlands\\
\And
{\bf Tom Claassen}  \\
Institute for Computing and Information Sciences\\
Radboud University Nijmegen\\
Nijmegen, The Netherlands
}

\begin{document}

\maketitle

\begin{abstract}
While feedback loops are known to play important roles in many complex systems, their existence is ignored in a large part of the causal discovery literature, as systems are typically assumed to be acyclic from the outset. When applying causal discovery algorithms designed for the acyclic setting on data generated by a system that involves feedback, one would not expect to obtain correct results. In this work, we show that---surprisingly---the output of the Fast Causal Inference (FCI) algorithm is correct if it is applied to observational data generated by a system that involves feedback. More specifically, we prove that for observational data generated by a simple and $\sigma$-faithful Structural Causal Model (SCM), FCI is sound and complete, and can be used to consistently estimate (i) the presence and absence of causal relations, (ii) the presence and absence of direct causal relations, (iii) the absence of confounders, and (iv) the absence of specific cycles in the causal graph of the SCM. We extend these results to constraint-based causal discovery algorithms that exploit certain forms of background knowledge, including the causally sufficient setting (e.g., the PC algorithm) and the Joint Causal Inference setting (e.g., the FCI-JCI algorithm).
\end{abstract}
\typoadd{This version fixes several typos in the published version.}

\section{INTRODUCTION}

Causal discovery, i.e., establishing the presence or absence of causal relationships between observed variables, is an important activity in
many scientific disciplines. Typical approaches to causal discovery from observational data are either score-based, or constraint-based (or a combination of the two). The more generally applicable constraint-based approach, which we focus on in this work, is based on exploiting 
information in 
conditional independences in the observed data to draw conclusions about the possible underlying causal structure. 

Although many systems of interest in various application domains involve feedback loops or other types of cyclic causal relationships (for example, in economical, biological, chemical, physical, control and climatological systems),
most of the existing literature on causal discovery from observational data ignores this and assumes from the outset that the
underlying causal system is acyclic. Nonetheless, several algorithms have been developed specifically for the
cyclic setting. For example, quite some work has been done for linear systems \citep[e.g.,][]{RiS99,LSR08,HEH10,HEH12,Rothenhausler_15}. 

More generally applicable are causal discovery algorithms that exploit conditional independence constraints,
without assuming certain restrictions on the parameterizations of the causal models (such as linearity).
Pioneering work in this area was done by \citet{RichardsonPhD1996}, resulting in the CCD algorithm,
the first constraint-based causal discovery algorithm shown to be applicable in a cyclic setting
\citep[see also ][]{Richardson96,RiS99}. It
was shown to be sound under the assumptions of causal sufficiency, the $d$-separation Markov property, 
and $d$-faithfulness. More recently, other algorithms that are sound under these assumptions
(except for the requirement of causal sufficiency) were proposed \citep{HEJ14,Strobl2018}.

However, it was already noted by \citet{Spirtes94,Spirtes95} that the $d$-separation Markov
property assumption can be too strong in general, and he proposed an alternative criterion, making use of the so-called
``collapsed graph'' construction. More recently, an alternative formulation in terms of the \emph{$\sigma$-separation}
criterion was introduced, and the corresponding Markov property was shown to hold in a very general setting 
\citep{ForreMooij_1710.08775}.
Whereas the Markov property based on $\sigma$-separation applies under mild assumptions, the stronger Markov
property based on $d$-separation is limited to more specific settings (e.g., continuous variables with linear relations, or
discrete variables, or the acyclic case) \citep{ForreMooij_1710.08775}.  
As discussed in \citep{ForreMooij_1710.08775,Bongers++_1611.06221v3}, the 
$\sigma$-separation Markov property 
seems appropriate for a wide class of cyclic structural causal models with non-linear functional relationships
between non-discrete variables, for example structural causal models corresponding to the equilibrium states of 
dynamical systems governed by random differential equations \citep{BongersMooij_1803.08784}.

Apart from a Markov property, constraint-based causal discovery algorithms need to make some type of faithfulness assumption.
A natural extension of the common faithfulness assumption used in the acyclic setting is obtained by
replacing $d$-separation by $\sigma$-separation, that we refer to as \emph{$\sigma$-faithfulness}. 
\citet{ForreMooij_UAI_18} proposed a constraint-based causal discovery algorithm that is sound and
complete, assuming the $\sigma$-separation Markov property in combination with the $\sigma$-faithfulness assumption.
However, their algorithm is limited in practice to about 5--7 variables because of the
combinatorial explosion in the number of possible causal graphs with increasing number of variables.
Interestingly, under the additional assumption of causal sufficiency,
the CCD algorithm is also sound under these assumptions \citep[as already noted in Section 4.5 of ][]{RichardsonPhD1996}.
Other causal discovery algorithms (LCD \citep{Cooper1997}, ICP \citep{ICP2016} and Y-structures \citep{ManiPhD2006}), 
all originally designed for the acyclic setting, have been shown to be sound also in the 
$\sigma$-separation setting \citep{Mooij++_JMLR_2020}.
The most general scenario (under the additional assumption of causal sufficiency, however) is addressed by the NL-CCD algorithm 
\citep[Chapter 4 in ][]{RichardsonPhD1996}, which was shown to be sound under the assumptions of the $\sigma$-separation 
Markov property together with the (weaker) $d$-faithfulness assumption.

One of the classic algorithms for constraint-based causal discovery is the Fast Causal Inference (FCI) algorithm \citep{SMR1995,SMR1999,Zhang2008_AI}. It was designed for the acyclic case, assuming the $d$-separation Markov property in combination with the $d$-faithfulness assumption.
Recently, it was observed that when run on data generated by cyclic causal models, the accuracy of FCI is actually comparable to its accuracy in the strictly
acyclic setting \citep[Figures 25, 26, 29, 31, 32 in][]{Mooij++_JMLR_2020}. This is surprising, as it is commonly believed that the application domain of FCI is limited to acyclic causal systems, and one would expect such serious model misspecification to result in glaringly incorrect results.

In this work, we show that when FCI is applied on data from a cyclic causal system that satisfies the $\sigma$-separation
Markov property and is $\sigma$-faithful, its output is still sound and complete. 
Furthermore, we derive criteria for how to read off various features from the partial ancestral graph output by FCI
(specifically, the absence or presence of ancestral relations, direct relations, cyclic relations and confounders).
This provides a practical causal discovery algorithm for that setting that is able
to handle hundreds or even thousands of variables as long as the underlying causal model is sparse enough, and that is
also applicable in the presence of latent confounders. 
It thus forms a significant improvement over the previous state-of-the-art in causal discovery for the $\sigma$-separation setting.  

The results we derive
in this work are not limited to FCI, but apply to any constraint-based causal discovery algorithm that solves the same task
as FCI does, i.e., that estimates the directed partial ancestral graph from conditional independences in the data,
e.g., FCI+ \citep{ClaassenMooijHeskes_UAI_13} and CFCI \citep{Colombo++2012}.
Our results therefore make constraint-based causal discovery in the presence of cycles as practical as it is in the acyclic case, without requiring any modifications of the algorithms.
Our work also provides the first characterization of the $\sigma$-Markov equivalence class of directed mixed graphs. 
We extend our results to variants of algorithms that exploit certain background knowledge,
for example, causal sufficiency \citep[e.g., the PC algorithm,][]{SGS2000} or the Joint Causal Inference framework 
\citep[e.g., the FCI-JCI algorithm, ][]{Mooij++_JMLR_2020}.
For simplicity, we assume no selection bias in this work, but we expect that our results can be extended to allow for that as well.

\section{PRELIMINARIES}

In Section~\ref{app:preliminaries} (Supplementary Material), we introduce our notation and terminology and provide the reader with a summary of the necessary definitions and results from the graphical causal modeling and discovery literature.
For more details, we refer the reader to the literature
\citep{Pearl2009,SGS2000,RichardsonSpirtes02,Zhang2006,Zhang2008_AI,Zhang2008_JMLR,Bongers++_1611.06221v3,ForreMooij_1710.08775}.
Here, we only give a short high-level overview of the key notions.

There exists a variety of graphical representations of causal models. Most popular are \emph{directed acyclic graphs (DAGs)}, 
presumably because of their simplicity. DAGs are appropriate under the assumptions of causal 
sufficiency (i.e., there are no latent common causes of the observed variables), acyclicity (absence of feedback loops) and 
no selection bias (i.e., there is no implicit conditioning on a common effect of the observed variables). DAGs have many
convenient properties, amongst which a Markov property (which has different equivalent formulations, the most prominent one 
being in terms of the notion of $d$-separation)
and a simple causal interpretation. A more general class of graphs are \emph{acyclic directed mixed graphs (ADMGs)}. 
These make use of
additional bidirected edges to represent latent confounding, and have a similarly convenient Markov property (sometimes
referred to as $m$-separation) and causal
interpretation. When also dropping the assumption of acyclicity (thereby allowing for feedback), one can make use of the more
general class of \emph{directed mixed graphs (DMGs)}. 
These graphs can be naturally associated with (possibly cyclic) structural causal
models (SCMs) and can represent feedback loops. 
The corresponding Markov properties and causal interpretation are more subtle \citep{Bongers++_1611.06221v3} than
in the acyclic case. Cyclic SCMs can be used, e.g., to describe the causal semantics of the equilibrium states of 
dynamical systems governed by random differential equations \citep{BongersMooij_1803.08784}.

In this work, we will restrict ourselves to the subclass of \emph{simple} SCMs, i.e., those SCMs for
which any subset of the structural equations has a unique solution for the corresponding endogenous variables in terms of the
other variables appearing in these equations. 
Simple SCMs admit (sufficiently weak) cyclic interactions but retain many of the 
convenient properties of acyclic SCMs \citep{Bongers++_1611.06221v3}.
 They are a special case of modular SCMs \citep{ForreMooij_1710.08775}.
In particular, they satisfy the $\sigma$-separation Markov property and their graphs have an intuitive causal interpretation.\footnote{The $\sigma$-separation criterion is very similar to the $d$-separation criterion, with the only difference
being that $\sigma$-separation has as an additional condition for a non-collider to block a path that it has to point to a node in a 
different strongly connected component. Two nodes in a DMG are said to be in the same strongly connected component 
if and only if they are both ancestor of each other.}

For acyclic constraint-based causal discovery, ADMGs provide a more fine-grained representation than necessary, because one can only 
recover the
Markov equivalence class of ADMGs from conditional independences in observational data. A less expressive class of graphs, \emph{maximal ancestral graphs (MAGs)}, was introduced by \citet{RichardsonSpirtes02}. Each ADMG induces a MAG and each MAG represents a set of ADMGs. The mapping from ADMG to MAG preserves the $d$-separations and the (non-)ancestral relations. 
Contrary to ADMGs, MAGs have at most a single edge connecting any pair of distinct variables.
One of the key properties that distinguishes MAGs from ADMGs is that Markov-equivalent MAGs have the same adjacencies.
In addition to being able to handle latent variables, MAGs can also represent implicit conditioning on a subset of the
variables, making use of undirected edges. Therefore, they can be used to represent both latent variables and selection bias.

It is often convenient when performing causal reasoning or discovery to be able to represent a set of hypothetical MAGs in a compact way. For these reasons, \emph{partial ancestral graphs (PAGs)} were introduced \citep{Zhang2006}.\footnote{PAGs
were originally introduced by \citet{RichardsonPhD1996} in order to represent the output of the CCD algorithm. It was conjectured by Richardson that PAGs could also be used to represent the output of the FCI algorithm, which was originally formulated in terms of Partially Oriented Inducing Path Graphs (POIPGs). This conjecture was proved subsequently by Spirtes. \citet[p.\ 102,][]{RichardsonPhD1996} notes: ``It is an open question whether or not the set of symbols is sufficiently rich to allow us to represent the
class of cyclic graphs with latent variables.'' In the present work we turned full circle by reinterpreting PAGs
as representing properties of DMGs, and have thereby answered this question affirmatively.}
The usual way to think about a PAG is as an object that represents a set of MAGs. 
The (Augmented) Fast Causal Inference (FCI) algorithm \citep{SMR1995,SMR1999,Zhang2008_AI} takes as input the conditional independences that hold in the data (assumed to be $d$-Markov and $d$-faithful w.r.t.\ a ``true'' ADMG), and outputs a PAG. 
As shown in seminal work  \citep{SMR1995,SMR1999,Ali++2005,Zhang2008_AI}, the FCI algorithm is sound and complete, and the PAG output by FCI represents the Markov equivalence class of the true ADMG. 

In this work, we will for simplicity assume no selection bias.
This means that we can restrict ourselves to MAGs without undirected edges, which we refer to as \emph{directed MAGs (DMAGs)}, and PAGs without undirected or circle-tail edges, which we refer to as \emph{directed PAGs} (DPAGs).

\begin{figure}\centering
\scalebox{0.7}{\begin{tikzpicture}
  \node (SCM) at (-3,3) {SCM};
  \node (DMG) at (-2,2) {DMG};
  \node (ADMG1) at (0.5,1)  {ADMG 1};
  \node (ADMG2) at (0.5,0)  {ADMG 2};
  \node (ADMG3) at (0.5,-1) {ADMG $n$};
  \node (DMAG1) at (3,1)  {DMAG 1};
  \node (DMAG2) at (3,0)  {DMAG 2};
  \node (DMAG3) at (3,-1) {DMAG $n$};
  \node (IM)    at (6,0)  {IM};
  \node (CDPAG) at (6,2)  {DPAG};
  \node   at (1,-0.5) {\dots};
  \node   at (3,-0.5) {\dots};
  \draw[arr] (SCM) to (DMG);
  \draw[arr,bend left] (DMG) to node [anchor=east,yshift=-2mm] {$\sigma$-separation} (IM);
  \draw[arr,bend right] (DMG) to (ADMG1);
  \draw[arr,bend right] (DMG) to (ADMG2);
  \draw[arr,bend right] (DMG) to node [anchor=east] {acyclifications} (ADMG3);
  \draw[arr] (ADMG1) to (DMAG1);
  \draw[arr] (ADMG2) to (DMAG2);
  \draw[arr] (ADMG3) to (DMAG3);
  \draw[arr] (DMAG1) to (IM);
  \draw[arr] (DMAG2) to (IM);
  \draw[arr] (DMAG3) to node [anchor=north,xshift=5mm,yshift=-2mm] {$d$-separation} (IM);
  \draw[arr] (IM) to node [anchor=west] {FCI} (CDPAG);
  \draw[arr,bend right] (CDPAG) to node [anchor=north] {contains} (DMG);
\end{tikzpicture}}
\caption{Relations between various representations.\label{fig:different_graphs}}
\end{figure}
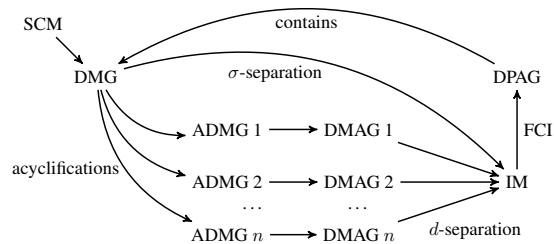

\section{EXTENSIONS TO THE CYCLIC SETTING}

The theory of MAGs and PAGs is rather intricate. A natural question is how this theory can be extended
when the assumption of acyclicity is dropped. This does not seem to be straightforward at first sight. 
An obvious approach would be to generalize the notion of MAGs by adding edge types that represent cycles. 
However, it would probably require a lot of effort to rederive and reformulate the known results about MAGs and PAGs in this more general setting.
In this work, we take another approach: we represent a (possibly cyclic) DMG directly by a DPAG.
In order to make this idea precise, we first need to extend the notion of inducing path to the cyclic setting.
Our strategy is illustrated in Figure~\ref{fig:different_graphs}.

\subsection{INDUCING PATHS}

We propose the following generalization of the notion of inducing path (Definition~\ref{def:inducing_path_acyclic}) to the $\sigma$-separation setting:
\begin{definition}\label{def:inducing_path}
Let $\C{G} = \langle \C{V}, \C{E}, \C{F} \rangle$ be directed mixed graph (DMG). 
An \emph{inducing path (walk) between two nodes $i,j \in \C{V}$} is a path (walk) in $\C{G}$ between $i$ and $j$ on which every 
  collider is in $\ansub{\C{G}}{\{i,j\}}$, and each \typoadd{non-endpoint} non-collider on the path (walk)\typorm{, except $i$ and  $j$,} only has outgoing directed edges to neighboring nodes on the path (walk) that lie in the same strongly connected component of $\C{G}$.
\end{definition}
This is indeed the proper generalization, since it has the following property.
\begin{proposition}
  Let $\C{G} = \langle \C{V},\C{E},\C{F} \rangle$ be a DMG and $i,j$ two distinct vertices in $\C{G}$. Then the following are equivalent:
  \begin{compactenum}[(i)]
    \item There is an inducing path in $\C{G}$ between $i$ and $j$;
    \item There is an inducing walk in $\C{G}$ between $i$ and $j$;
    \item $\sigmacon{i}{j}{Z}{\C{G}}$ for all $Z \subseteq \C{V} \setminus \{i,j\}$.
  \end{compactenum}
\end{proposition}
\begin{proof}
  The proof is similar to that of Theorem 4.2 in \citet{RichardsonSpirtes02}.

  (i) $\implies$ (ii) is trivial.

  (ii) $\implies$ (iii): Assume the existence of an inducing walk between $i$ and $j$ in $\C{G}$. 
  Let $Z \subseteq \C{V}\setminus\{i,j\}$. 
  Consider a walk $\mu$ in $\C{G}$ between $i$ and $j$ with the fewest number of colliders, with the property that all colliders on $\mu$ are in $\ansub{\C{G}}{\{i,j\} \cup Z}$, and each non-endpoint non-collider on $\mu$ is not in $Z$ or points only to nodes in the same strongly connected component of $\C{G}$. 
  Such a walk must exist (because the inducing walk satisfies that property). 
  We now show that all colliders on $\mu$ must be in $\ansub{\C{G}}{Z}$. 
  Suppose on the contrary the existence of a collider $k$ on $\mu$ that is not ancestor of $Z$. 
  It is either ancestor of $i$ or of $j$, by assumption.
  Without loss of generality, assume the latter.
  Then there is a directed path $\pi$ from $k$ to $j$ in $\C{G}$ that does not pass through any node of $Z$. 
  Then the subwalk of $\mu$ between $i$ and $k$ can be concatenated with the directed path $\pi$
  into a walk between $i$ and $j$ that has the property, but has fewer colliders than $\mu$: a contradiction. 
  Therefore, $\mu$ is $\sigma$-open given $Z$.
  Hence, $i$ and $j$ are $\sigma$-connected given $Z$.

  (iii) $\implies$ (i): Suppose that $i$ and $j$ are $\sigma$-connected given any $Z \subseteq \C{V} \setminus \{i,j\}$.
  In particular, $i$ and $j$ are $\sigma$-connected given $Z^* = \ansub{\C{G}}{\{i,j\}} \setminus \{i,j\}$. 
  Let $\pi$ be a path between $i$ and $j$ that is $\sigma$-open given $Z^*$.
  We show that $\pi$ must be an inducing path. 
  First, all colliders on $\pi$ are in $\ansub{\C{G}}{Z^*}$ and hence in $\ansub{\C{G}}{\{i,j\}}$. 
  Second, let $k$ be any non-endpoint non-collider $k$ on $\pi$.
  Then there must be a directed subwalk of $\pi$ starting at $k$ that ends either at the first collider on $\pi$ next to $k$ or at an end node of $\pi$, and hence $k$ must be in $Z^*$. 
  Since $\pi$ is $\sigma$-open given $Z^*$, $k$ can only point to nodes in the same strongly connected component of $\C{G}$.
  Hence, all non-endpoint non-colliders on $\pi$ can only point to nodes in the same strongly connected component of $\C{G}$.
\end{proof}
In words: there is an inducing walk (or path) between two nodes in a DMG if and only if the two nodes cannot be $\sigma$-separated by any subset of nodes that does not contain either of the two nodes.

\subsection{REPRESENTING DMGs BY DPAGs}

The following definition forms the key to our approach.
\begin{definition}\label{def:dpag_contains_dmg}
  Let $\C{P}$ be a DPAG and $\C{G}$ a DMG, both with vertex set $\C{V}$. 
  We say that \emph{$\C{P}$ contains $\C{G}$} if all of the following hold:
  \begin{compactenum}[(i)]
    \item two vertices $i,j$ are adjacent in $\C{P}$ if and only if there is an inducing path between $i,j$ in $\C{G}$;
    \item if $i \sto j$ in $\C{P}$ (i.e., $i \to j$ in $\C{P}$ or $i \cto j$ in $\C{P}$ or $i \oto j$ in $\C{P}$), then $j \notin \ansub{\C{G}}{i}$;
    \item if $i \to j$ in $\C{P}$ then $i \in \ansub{\C{G}}{j}$.
  \end{compactenum}
\end{definition}
It is only a slight variation on how PAGs are traditionally interpreted, and agrees with the traditional (acyclic)
interpretation when restricting the DMGs to be acyclic.

\subsection{ACYCLIFICATIONS}

Inspired by the ``collapsed graph'' construction of 
\citet{Spirtes94,Spirtes95}, \citet{ForreMooij_1710.08775} introduced a notion of \emph{acyclification} for a class of graphical causal models termed HEDGes, but the same concept can be defined for DMGs, which we will do here. 
\begin{definition}\label{def:acyclification}
  Given a DMG $\C{G} = \langle \C{V}, \C{E}, \C{F} \rangle$. An \emph{acyclification} of $\C{G}$ is an ADMG $\C{G}' = \langle \C{V}, \C{E}', \C{F}' \rangle$ with 
  \begin{compactenum}[(i)]
    \item the same nodes $\C{V}$;
    \item for any pair of nodes $\{i,j\}$ such that $i \not\in\sccsub{\C{G}}{j}$:
      \begin{compactenum}
         \item $i \to j \in \C{E}'$ iff there exists a node $j'$ such that $j' \in \sccsub{\C{G}}{j}$ and $i \to j' \in \C{E}$;
         \item $i \oto j \in \C{F}'$ iff there \typo{exists a node $k$ such that $k \in \sccsub{\C{G}}{j}$ and $i \oto k \in \C{F}$}{exist nodes $i',j'$ such that $i' \in \sccsub{\C{G}}{i}$, $j' \in \sccsub{\C{G}}{j}$ and $i' \oto j' \in \C{F}$};
      \end{compactenum}\label{def:acyclification_inter_scc}
    \item for any pair of distinct nodes $\{i,j\}$ such that $i \in\sccsub{\C{G}}{j}$: $i \to j \in \C{E}'$ or $i \ot j \in \C{E}'$ or $i \oto j \in \C{F}'$.\label{def:acyclification_intra_scc}
  \end{compactenum}
\end{definition}
In words: all strongly connected components are made fully-connected, edges between strongly connected components are preserved, and any edge into a node in a strongly connected component must be copied and made adjacent to all nodes in the strongly connected component. Note that a DMG may have multiple acyclifications. An example is given in Figure~\ref{fig:acyclification}.

\begin{figure*}
\scalebox{0.7}{\begin{tikzpicture}
  \begin{scope}
    \node at (-1.5,2.5) {$\C{G}$:};
    \node[var] (X1) at ( 0.00, 2.25) {$X_1$};
    \node[var] (X2) at (-2.00, 0.00) {$X_2$};
    \node[var] (X3) at (-0.75, 0.75) {$X_3$};
    \node[var] (X4) at (-0.75,-0.75) {$X_4$};
    \node[var] (X5) at ( 0.75, 0.75) {$X_5$};
    \node[var] (X6) at ( 0.75,-0.75) {$X_6$};
    \node[var] (X8) at ( 2.25,-0.75) {$X_7$};
    \node[var] (X9) at (-2.5,  1.50) {$X_8$};
    \node[var] (X10) at (-1.5, 1.50) {$X_9$};
    \node[var] (X11) at ( 1.5, 2.25) {$X_{10}$};
    \draw[arr] (X11) to (X1);
    \draw[arr] (X9) to (X2);
    \draw[arr] (X10) to (X2);
    \draw[biarr] (X1) edge (X3);
    \draw[arr] (X2) edge (X4);
    \draw[arr] (X3) edge (X4);
    \draw[arr] (X4) edge (X6);
    \draw[arr] (X6) edge (X5);
    \draw[arr] (X5) edge (X3);
    \draw[arr] (X6) to (X8);
  \end{scope}
  \begin{scope}[xshift=6cm]
    \node at (-1.5,2.5) {$\C{G}^{\acy}$:};
    \node[var] (X1) at ( 0.00, 2.25) {$X_1$};
    \node[var] (X2) at (-2.25, 0.00) {$X_2$};
    \node[var] (X3) at (-0.75, 0.75) {$X_3$};
    \node[var] (X4) at (-0.75,-0.75) {$X_4$};
    \node[var] (X5) at ( 0.75, 0.75) {$X_5$};
    \node[var] (X6) at ( 0.75,-0.75) {$X_6$};
    \node[var] (X8) at ( 2.25,-0.75) {$X_7$};
    \node[var] (X9) at (-2.5,  1.50) {$X_8$};
    \node[var] (X10) at (-1.5, 1.50) {$X_9$};
    \node[var] (X11) at ( 1.5, 2.25) {$X_{10}$};
    \draw[arr] (X11) to (X1);
    \draw[arr] (X9) to (X2);
    \draw[arr] (X10) to (X2);
    \draw[biarr] (X1) edge (X3);
    \draw[biarr,bend left=11] (X1) edge (X4);
    \draw[biarr] (X1) edge (X5);
    \draw[biarr,bend right=11] (X1) edge (X6);
    \draw[arr] (X2) edge (X3);
    \draw[arr] (X2) edge (X4);
    \draw[arr,bend right=11] (X2) edge (X5);
    \draw[arr,bend left=11] (X2) edge (X6);
    \draw[biarr] (X3) edge (X4);
    \draw[biarr] (X4) edge (X6);
    \draw[biarr] (X6) edge (X5);
    \draw[biarr] (X5) edge (X3);
    \draw[biarr] (X3) edge (X6);
    \draw[biarr] (X4) edge (X5);
    \draw[arr] (X6) to (X8);
  \end{scope}
  \begin{scope}[xshift=12cm]
    \node at (-1.5,2.5) {$\C{G}'$:};
    \node[var] (X1) at ( 0.00, 2.25) {$X_1$};
    \node[var] (X2) at (-2.25, 0.00) {$X_2$};
    \node[var] (X3) at (-0.75, 0.75) {$X_3$};
    \node[var] (X4) at (-0.75,-0.75) {$X_4$};
    \node[var] (X5) at ( 0.75, 0.75) {$X_5$};
    \node[var] (X6) at ( 0.75,-0.75) {$X_6$};
    \node[var] (X8) at ( 2.25,-0.75) {$X_7$};
    \node[var] (X9) at (-2.5,  1.50) {$X_8$};
    \node[var] (X10) at (-1.5, 1.50) {$X_9$};
    \node[var] (X11) at ( 1.5, 2.25) {$X_{10}$};
    \draw[arr] (X11) to (X1);
    \draw[arr] (X9) to (X2);
    \draw[arr] (X10) to (X2);
    \draw[biarr] (X1) edge (X3);
    \draw[biarr,bend left=11] (X1) edge (X4);
    \draw[biarr] (X1) edge (X5);
    \draw[biarr,bend right=11] (X1) edge (X6);
    \draw[arr] (X2) edge (X3);
    \draw[arr] (X2) edge (X4);
    \draw[arr,bend right=11] (X2) edge (X5);
    \draw[arr,bend left=11] (X2) edge (X6);
    \draw[arr] (X3) edge (X4);
    \draw[arr] (X4) edge (X6);
    \draw[arr] (X5) edge (X6);
    \draw[arr] (X3) edge (X5);
    \draw[arr] (X3) edge (X6);
    \draw[arr] (X4) edge (X5);
    \draw[arr] (X6) to (X8);
  \end{scope}
  \begin{scope}[xshift=18cm]
    \node at (-1.5,2.5) {$\C{P}$:};
    \node[var] (X1) at ( 0.00, 2.25) {$X_1$};
    \node[var] (X2) at (-2.25, 0.00) {$X_2$};
    \node[var] (X3) at (-0.75, 0.75) {$X_3$};
    \node[var] (X4) at (-0.75,-0.75) {$X_4$};
    \node[var] (X5) at ( 0.75, 0.75) {$X_5$};
    \node[var] (X6) at ( 0.75,-0.75) {$X_6$};
    \node[var] (X8) at ( 2.25,-0.75) {$X_7$};
    \node[var] (X9) at (-2.5,  1.50) {$X_8$};
    \node[var] (X10) at (-1.5, 1.50) {$X_9$};
    \node[var] (X11) at ( 1.5, 2.25) {$X_{10}$};
    \draw[carr] (X11) to (X1);
    \draw[carr] (X9) to (X2);
    \draw[carr] (X10) to (X2);
    \draw[biarr] (X1) edge (X3);
    \draw[biarr,bend left=11] (X1) edge (X4);
    \draw[biarr] (X1) edge (X5);
    \draw[biarr,bend right=11] (X1) edge (X6);
    \draw[arr] (X2) edge (X3);
    \draw[arr] (X2) edge (X4);
    \draw[arr,bend right=11] (X2) edge (X5);
    \draw[arr,bend left=11] (X2) edge (X6);
    \draw[bicarr] (X3) edge (X4);
    \draw[bicarr] (X4) edge (X6);
    \draw[bicarr] (X5) edge (X6);
    \draw[bicarr] (X3) edge (X5);
    \draw[bicarr] (X3) edge (X6);
    \draw[bicarr] (X4) edge (X5);
    \draw[arr] (X6) to (X8);
  \end{scope}
\end{tikzpicture}}
\caption{From left to right: Directed mixed graph $\C{G}$, two of its acyclifications ($\C{G}^\acy$ and $\C{G}'$) and the DPAG output by FCI $\C{P} = \PAGFCI(\IM_\sigma(\C{G})) = \PAGFCI(\IM_d(\C{G}')) = \PAGFCI(\IM_d(\C{G}^\acy))$.\label{fig:acyclification}}
\end{figure*}
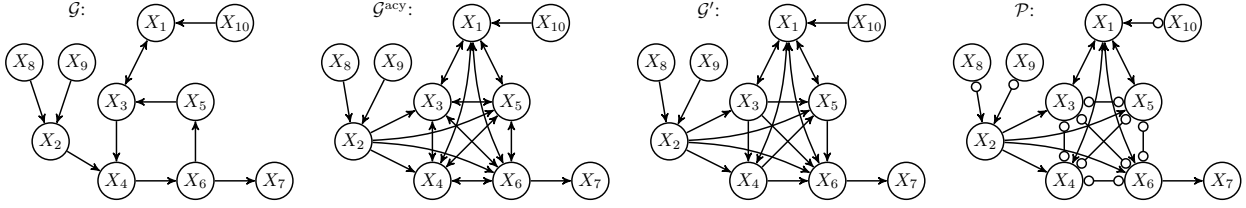

All acyclifications share certain ``spurious'' edges: the additional incoming directed and adjacent bidirected edges connecting nodes of two different strongly connected components. These have no \emph{causal} interpretation but are necessary to correctly represent the $\sigma$-separation properties as $d$-separation properties. 
The skeleton of any acyclification $\C{G}'$ of $\C{G}$ equals the skeleton of $\C{G}$
plus additional spurious adjacencies: the edges $i \ttt j$ with $i \sto k$ and $k \in \sccsub{\C{G}}{j}$,
and the edges $i \ttt j$ with $i \in \sccsub{\C{G}}{j}$ where $i$ and $j$ are not adjacent in $\C{G}$.
These ``spurious edges'' added in any acyclification of a DMG $\C{G}$ correspond with (non-trivial) inducing paths in $\C{G}$.

The ``raison d'\^etre'' for acyclifications is that they are $\sigma$-separation-equivalent to the original DMG, i.e., their
$\sigma$-independence models agree:
\begin{proposition}\label{prop:Markov_equivalence_acyclification}
  For any DMG $\C{G}$ and any acyclification $\C{G}'$ of $\C{G}$, $\IM_\sigma(\C{G}) = \IM_\sigma(\C{G}') = \IM_d(\C{G}')$.
\end{proposition}
\begin{proof}
  This follows from Theorem 2.8.3 in \citep{ForreMooij_1710.08775}.
\end{proof}
One particular acyclification that we will make use of repeatedly will be denoted $\C{G}^{\acy}$, and is obtained by
replacing all strongly connected components of $\C{G}$ by fully-connected bidirected components without any directed edges
(i.e., if $i \in \sccsub{\C{G}}{j}$ then $i \oto j$ in $\C{G}'$, but neither $i \to j$ nor $j \to i$ in $\C{G}'$).
Another useful set of acyclifications is obtained by replacing all strongly connected components of $\C{G}$ by 
arbitrary fully-connected DAGs, and optionally adding an arbitrary set of bidirected edges.
Other important properties of acyclifications are:
\begin{proposition}\label{prop:acyclification}
  Let $\C{G}$ be a DMG and $i,j$ two nodes in $\C{G}$.
  \begin{compactenum}[(i)]
  \item If $i \in \ansub{\C{G}}{j}$ then there exists an acyclification $\C{G}'$ of $\C{G}$ with $i \in \ansub{\C{G}'}{j}$; \label{prop:acyclification_anc}
    \item If $i \notin \ansub{\C{G}}{j}$ then $i \notin \ansub{\C{G}'}{j}$ for all acyclifications $\C{G}'$ of $\C{G}$; \label{prop:acyclification_non_anc}
    \item There is an inducing path between $i$ and $j$ in $\C{G}$ if and only if there is an inducing path between $i$ and $j$ in $\C{G}'$ for any acyclification $\C{G}'$ of $\C{G}$.\label{prop:acyclification_inducing_path}
  \end{compactenum}
\end{proposition}
\begin{proof}
(\ref{prop:acyclification_anc})
If $i \in \ansub{\C{G}}{j}$, then there exists a directed path from $i$ to $j$ in $\C{G}$. Any such directed path visits each strongly connected component of $\C{G}$ at most once. We can choose an acyclification $\C{G}'$ of $\C{G}$ with a suitably chosen DAG on each strongly connected component, in which we can take the shortcut $k \to l$ instead of each longest subpath $k \to u_1 \to \dots \to u_n \to l$ that consists entirely of nodes within a single strongly connected component of $\C{G}$. This yields a directed path from $i$ to $j$ in $\C{G}'$.

(\ref{prop:acyclification_non_anc}) Let $\C{G}'$ be an acyclification of $\C{G}$. Each directed edge $k \to l$ in $\C{G}'$ is either in $\C{G}$ or corresponds with $k \in \ansub{\C{G}}{l}$. Hence all ancestral relations in $\C{G}'$ must be present in $\C{G}$ as well. 





(\ref{prop:acyclification_inducing_path}) Follows directly from the separation properties of inducing paths and Proposition~\ref{prop:Markov_equivalence_acyclification}.
\end{proof}

\subsection{SOUNDNESS AND COMPLETENESS}

In the acyclic setting, the FCI algorithm was shown to be sound and complete \citep{Zhang2008_AI}.
The notion of acyclifications, together with their elementary properties
(Propositions \ref{prop:Markov_equivalence_acyclification} and \ref{prop:acyclification}) allows us to easily extend these soundness and completeness results to the $\sigma$-separation setting (allowing for cycles).

Consider FCI as a mapping $\PAGFCI$ from independence models (on variables $\C{V}$) to DPAGs (with vertex set $\C{V}$),
which maps the independence model of a DMG $\C{G}$ to the DPAG $\PAGFCI(\IM_\sigma(\C{G}))$.
\begin{theorem}\label{theo:fci_sound_complete}
In the $\sigma$-separation setting (but without selection bias), FCI is
\begin{compactenum}[(i)]
    \item \emph{sound}: for all DMGs $\C{G}$, $\PAGFCI(\IM_\sigma(\C{G}))$ contains $\C{G}$;\label{theo:fci_sound_complete_sound}
    \item \emph{arrowhead complete}: for all DMGs $\C{G}$: if $i \notin \ansub{\tilde{\C{G}}}{j}$ for any DMG $\tilde{\C{G}}$ that is $\sigma$-Markov equivalent to $\C{G}$, then there is an arrowhead $i \ots j$ in $\PAGFCI(\IM_\sigma(\C{G}))$;
    \item \emph{tail complete}: for all DMGs $\C{G}$, if $i \in \ansub{\tilde{\C{G}}}{j}$ in any DMG $\tilde{\C{G}}$ that is $\sigma$-Markov equivalent to $\C{G}$, then there is a tail $i \to j$ in $\PAGFCI(\IM_\sigma(\C{G}))$;
    \item \emph{Markov complete}: for all DMGs $\C{G}_1$ and $\C{G}_2$, $\C{G}_1$ is $\sigma$-Markov equivalent to $\C{G}_2$ iff $\PAGFCI(\IM_\sigma(\C{G}_1)) = \PAGFCI(\IM_\sigma(\C{G}_2))$.
\end{compactenum}
\end{theorem}
\begin{proof}
  The main idea is the following (see also Figure~\ref{fig:different_graphs}). For all DMGs $\C{G}$, $\IM_\sigma(\C{G}) = \IM_d(\C{G}')$ for any acyclification $\C{G}'$ of $\C{G}$ (Proposition~\ref{prop:Markov_equivalence_acyclification}). 
  Hence FCI maps any acyclification $\C{G}'$ of $\C{G}$ to the same DPAG $\PAGFCI(\IM_\sigma(\C{G}))$, and thereby any conclusion we draw about these acyclifications can be transferred back to a conclusion about $\C{G}$ by means of Proposition~\ref{prop:acyclification}.

  To prove soundness, let $\C{G}$ be a DMG and let $\C{P} = \PAGFCI(\IM_\sigma(\C{G}))$. 
  The acyclic soundness of FCI means that for all ADMGs $\C{G}'$, $\PAGFCI(\IM_d(\C{G}'))$ contains $\C{G}'$. 
  Hence, by Proposition~\ref{prop:Markov_equivalence_acyclification}, $\C{P}$ contains $\C{G}'$ for all acyclifications $\C{G}'$ of $\C{G}$. 
  But then $\C{P}$ must contain $\C{G}$:
  \begin{itemize} 
    \item if two vertices $i,j$ are adjacent in $\C{P}$ then there is an inducing path between $i,j$ in any acyclification of $\C{G}$, which holds if and only if there is an inducing path between $i,j$ in $\C{G}$ (Proposition~\ref{prop:acyclification}(\ref{prop:acyclification_inducing_path});
    \item if $i \sto j$ in $\C{P}$, then $j \notin \ansub{\C{G}'}{i}$ for any acyclification $\C{G}'$ of $\C{G}$, and hence $j \notin \ansub{\C{G}}{i}$ (Proposition~\ref{prop:acyclification}(\ref{prop:acyclification_anc}));
    \item if $i \to j$ in $\C{P}$, then $i \in \ansub{\C{G}'}{j}$ for all acyclifications $\C{G}'$ of $\C{G}$, and hence $i \in \ansub{\C{G}}{j}$ (Proposition~\ref{prop:acyclification}(\ref{prop:acyclification_non_anc})).
  \end{itemize}

  To prove arrowhead completeness, let $\C{G}$ be a DMG and suppose that $i \in \ansub{\tilde{\C{G}}}{j}$ in any DMG $\tilde{\C{G}}$ that is $\sigma$-Markov equivalent to $\C{G}$.
  Since $\C{G}^\acy$ is $\sigma$-Markov equivalent to $\C{G}$, this implies in particular that for all ADMGs $\tilde{\C{G}}$ that are $d$-Markov equivalent to $\C{G}^\acy$, $i \in \ansub{\tilde{\C{G}}}{j}$.
  Because of the acyclic arrowhead completeness of FCI, there must be an arrowhead $i \ots j$ in $\PAGFCI(\IM_d(\C{G}^\acy)) = \PAGFCI(\IM_\sigma(\C{G}))$.
  Tail completeness is proved similarly.

  To prove Markov completeness: Proposition~\ref{prop:Markov_equivalence_acyclification}
  implies both $\IM_\sigma(\C{G}_1) = \IM_d(\C{G}_1^\acy)$ and
  $\IM_\sigma(\C{G}_2) = \IM_d(\C{G}_2^\acy)$. From the acyclic Markov completeness of FCI
  (Proposition~\ref{prop:PAG=MEC} in the Supplementary Material), it then follows that
  $\C{G}_1^\acy$ must be $d$-Markov equivalent to $\C{G}_2^\acy$, and
  hence $\C{G}_1$ must be $\sigma$-Markov equivalent to $\C{G}_2$.

  Alternatively, the statement of this theorem can be seen to be a special case of Theorem~\ref{theo:generalizing_soundness_completeness}, applied with the trivial background knowledge $\Psi(\C{G}) = 1$ for all DMGs $\C{G}$, to $\Phi : \C{G} \mapsto \PAGFCI(\IM_\sigma(\C{G}))$, combined with the known (acyclic) soundness and completeness results of FCI \citep{Zhang2008_AI}. 
  Note here that the trivial background knowledge $\Psi(\C{G}) = 1$ satisfies Assumption~\ref{ass:acybk} as follows immediately from Proposition~\ref{prop:acyclification}.
\end{proof}
Note that these definitions of soundness and completeness reduce to their acyclic counterparts \citep{Zhang2008_AI} when restricting to ADMGs.
In particular, the soundness and Markov completeness properties together imply that the DPAG $\PAGFCI(\IM_\sigma(\C{G}))$ output by FCI, when given as input the $\sigma$-independence model of a DMG $\C{G}$, represents the $\sigma$-Markov equivalence class of $\C{G}$.
In other words, FCI provides a characterization of the $\sigma$-Markov equivalence class of a DMG.
This is, to the best of our knowledge, the first such characterization.

In order to read off the independence model from the DPAG $\PAGFCI(\IM_\sigma(\C{G}))$, one can follow the same procedure as in the acyclic case: first construct a representative DMAG (for details, see \citet{Zhang2008_AI}) and then apply the $d$-separation criterion to this DMAG. 
While the soundness of FCI (Theorem~\ref{theo:fci_sound_complete}(\ref{theo:fci_sound_complete_sound})) allows us to read off
some (non-)ancestral relations from the DPAG output by FCI, this is by far not all causal information that is identifiable from the $\sigma$-Markov equivalence class. 
In the following sections, we will discuss how various causal features can be identified from DPAGs.

\subsection{IDENTIFIABLE (NON-)ANCESTRAL RELATIONS}

We make use here of the following definition \citep{Zhang2008_AI}:
\begin{definition}\label{def:pdpath}
  A path $v_0,e_1,v_1,\dots,v_n$ between nodes $v_0$ and $v_n$ in a DPAG $\C{P}$ is called a \emph{possibly directed path from $v_0$ to $v_n$} if for each $i=1,\dots,n$, the edge $e_i$ between $v_{i-1}$ and $v_i$ is not into $v_{i-1}$ (i.e., is of the form $v_{i-1} \ctc v_i$, $v_{i-1} \cto v_i$, or $v_{i-1} \to v_i$). 
  The path is called \emph{uncovered} if every subsequent triple is unshielded, i.e., $v_i$ and $v_{i-2}$ are not adjacent in $\C{P}$ for $i=2,\dots,n$.
\end{definition}

\citet{Zhang2006} conjectured the soundness and completess of a criterion to read off all invariant ancestral relations from a complete DPAG, i.e., to identify the ancestral relations that are present in all Markov equivalent ADMGs that are represented by a complete DPAG.
\citet{Roumpelaki++_UAIWS_16} proved soundness of the criterion.\footnote{They also claim to have proved completeness, but their proof is flawed: the last part of the proof that aims to prove that $u,v$ are non-adjacent appears to be incomplete.}
We extend Theorem 3.1 in \citep{Roumpelaki++_UAIWS_16} to DPAGs and DMGs:
\begin{proposition}\label{prop:cdpag_ancestors_dmg}
  Let $\C{G}$ be a DMG, and let $\C{P}$ be a DPAG that contains $\C{G}$, and such that all unshielded triples in $\C{P}$ have been oriented according to FCI rule $\C{R}0$ \citep{Zhang2008_AI} using $\IM_\sigma(\C{G})$.
  For two nodes $i \ne j \in \C{P}$: If 
  \begin{compactitem}
    \item there is a directed path from $i$ to $j$ in $\C{P}$, or 
    \item there exist uncovered possibly directed paths (see Definition~\ref{def:pdpath}) from $i$ to $j$ in $\C{P}$ of the form $i,u,\dots,j$ and $i,v,\dots,j$ such that $u,v$ are distinct non-adjacent nodes in $\C{P}$, 
  \end{compactitem}
  then $i \in \ansub{\C{G}}{j}$, i.e., $i$ is ancestor of $j$ according to $\C{G}$.
\end{proposition}
\begin{proof}
  If there is a directed path from $i$ to $j$ in $\C{P}$, say $v_1 \to \dots \to v_n$ with $v_1 = i$ and $v_n = j$, then $v_m \in \ansub{\C{G}}{v_{m+1}}$ for all $m = 1,\dots,n-1$. Hence $i \in \ansub{\C{G}}{j}$.

  Otherwise, assume that there exist uncovered possibly directed paths from $i$ to $j$ in $\C{P}$ of the form $i,u,\dots,j$ and $i,v,\dots,j$ such that $u,v$ are distinct and non-adjacent in $\C{P}$. 
  If $i \in \ansub{\C{G}}{u}$, the path $i,u,\dots,j$ must actually correspond with a directed path from $i$ to $j$ in $\C{G}$, because otherwise it would contain unshielded colliders that were not oriented, contradicting the assumptions. 
  If $i \notin \ansub{\C{G}}{u}$ instead, one obtains that $i \in \ansub{\C{G}}{v}$ to avoid an unshielded collider $u \sto i \ots v$ in $\C{P}$ that was not oriented. 
  Hence the path $i,v,\dots,j$ must correspond with a directed path from $i$ to $j$ in $\C{G}$, because otherwise it would contain unshielded colliders that were not oriented, contradicting the assumptions.
  In both cases, $i \in \ansub{\C{G}}{j}$.
\end{proof}
As an example, from the (complete) DPAG in Figure~\ref{fig:acyclification} it follows that $X_2 \in \ansub{\C{G}}{X_4}$, and $X_2 \in \ansub{\C{G}}{X_7}$. 

\citet[p.\ 137]{Zhang2006} provides a sound and complete criterion to read off definite non-ancestors from a complete DPAG,
assuming acyclicity. We can directly extend this criterion to DPAGs and DMGs:
\begin{proposition}\label{prop:cdpag_nonancestors}
  Let $\C{G}$ be a DMG, and let $\C{P}$ be a DPAG that contains $\C{G}$.
  For two nodes $i\ne j \in \C{P}$:
  if there is no possibly directed path from $i$ to $j$ in $\C{P}$ then $i \notin \ansub{\C{G}}{j}$.
\end{proposition}
\begin{proof}
  If $i \in \ansub{\C{G}}{j}$, then the directed path from $i$ to $j$ must correspond with a possibly directed path in $\C{P}$.
\end{proof}
As an example, from the DPAG in Figure~\ref{fig:acyclification} we can read off that $X_8$ cannot be ancestor of $X_1$ in $\C{G}$, nor the other way around.
However, $X_3 \ctc X_6 \to X_7$ is a possibly directed path in the DPAG, and so $X_3$ may be (and in this case is) ancestor of $X_7$ in $\C{G}$.

\subsection{IDENTIFIABLE NON-CONFOUNDED PAIRS}

While in ADMGs and DMGs confounding is indicated by bidirected edges, in DPAGs confounding can also ``hide'' behind
directed edges. The following notion is of key importance in this regard:
\begin{definition}[\citet{Zhang2008_JMLR}]\label{def:visible}
A directed edge $i \to j$ in a DMAG is said to be \emph{visible} if there is a node
$k$ not adjacent to $j$, such that either there is an edge between $k$ and $i$ that is into $i$, or there is a
collider path between $k$ and $i$ that is into $i$ and every collider on the path is a parent of $j$. Otherwise
$i \to j$ is said to be \emph{invisible}.
The same notion applies to a DPAG, but is then called \emph{definitely visible} (and its negation \emph{possibly invisible}).
\end{definition}
For example, in the DPAG in Figure~\ref{fig:acyclification}, edge $X_6 \to X_7$
is definitely visible (by virtue of $X_2 \to X_6$), as are all edges
$X_2 \to \{X_3,X_4,X_5,X_6\}$ (by virtue of $X_8 \cto X_2$, or $X_9 \cto X_2$).

The notion of (in)visibility is closely related with confounding, as shown in Lemma 9 and 10 in \citet{Zhang2008_JMLR}.
To generalize this, we make use of the following Lemma.
\begin{lemma}\label{lemm:inducing_walk_into}
  Let $\C{P}$ be a DPAG that contains DMG $\C{G}$, and let $k \sto i$ be an edge in $\C{P}$ that is into $i$.
  Then there exists an inducing walk in $\C{G}$ between $k$ and $i$ that is into $i$.
  If $k \oto i$ in $\C{P}$, then there exists an inducing walk in $\C{G}$ between $k$ and $i$ that is both into $k$ and into $i$.
\end{lemma}
\begin{proof}
  If $k \sto i$ in $\C{P}$, then there 
  exists an inducing walk between $k$ and $i$ in $\C{G}$ because $k$ and $i$ are adjacent in $\C{P}$ and $\C{P}$ contains $\C{G}$.
  If this inducing walk were out of $i$, it would be of the form $k \ldots \sto u_n \ot u_{n-1} \ot \dots \ot u_1 \ot i$, where $u_n$ is the first collider on the walk that one encounters when following the directed edges out of $i$. 
  $u_n$ must be ancestor of $i$ or $k$ in $\C{G}$, and it cannot be ancestor of $k$ (because then $i$ would be ancestor of $k$, contradicting the orientation $k \sto i$ in $\C{P}$), hence it must be ancestor of $i$. 
  Thus there exists a walk $k \ldots \sto u_n \to \dots \to i$ in $\C{G}$ where we replaced the subwalk $u_n \ot u_{n-1} \ot \dots \ot i$ by a directed path from $u_n$ to $i$ \typoadd{that is entirely in $\sccsub{\C{G}}{i}$}.
  It is clear that this is an inducing \typo{path}{walk} in $\C{G}$ between $k$ and $i$ that is into $i$.

  If $k \oto i$ in $\C{P}$, then by similar reasoning, we obtain an inducing \typo{path}{walk} in $\C{G}$ between $k$ and $i$ that is into $k$ as well as into $i$.
\end{proof}
This allows us to generalize Lemma 9 in \citep{Zhang2008_JMLR} to the cyclic setting (with almost identical proof).
\begin{lemma}\label{lemm:Zhang_invisible}
  Let $\C{P}$ be a DPAG, and $i \to j$ a directed edge in $\C{P}$.
  If $i \to j$ is definitely visible in $\C{P}$, then for all DMGs $\C{G}$ contained in $\C{P}$, there exists no inducing walk between $i$ and $j$ in $\C{G}$ that is into $i$. 
\end{lemma}
\begin{proof}
  Suppose $\C{G}$ is a DMG contained in $\C{P}$. 
  Then $i \in \ansub{\C{G}}{j}$ and there exists an inducing walk between $i$ and $j$ in $\C{G}$. 
  We will prove the contrapositive. 
  Assume that there exists an inducing walk in $\C{G}$ between $i$ and $j$ that is into $i$.
  Let $k$ be another vertex in $\C{P}$. 

  If $k \sto i$ in $\C{P}$, then there is an inducing walk between $k$ and $i$ in $\C{G}$ that is into $i$ by Lemma~\ref{lemm:inducing_walk_into}.
 
  If there is a collider path $\pi$ in $\C{P}$ from $k$ to $i$ that is into $i$ and such that every non-endpoint vertex on the walk is parent of $j$ in $\C{P}$, then there is an inducing walk between $k$ and $i$ in $\C{G}$ that is into $i$.
  For each pair of adjacent vertices $(v_i,v_{i+1})$ on $\pi$, Lemma~\ref{lemm:inducing_walk_into} gives the existence of an inducing walk in $\C{G}$ between $v_i$ and $v_{i+1}$ that is into $v_{i+1}$, and also into $v_i$ unless possibly $v_i = k$.
  Because each vertex other than $k$ on $\pi$ is ancestor in $\C{G}$ of $j$, all these inducing walks can be concatenated into one inducing walk in $\C{G}$ between $k$ and $i$ that is into $i$.

  Concatenating the inducing walk between $k$ and $i$ that is into $i$ with the inducing walk between $i$ and $j$ that is into $i$ we obtain an inducing walk between $k$ and $j$ (note that $i$ becomes a collider that is ancestor of $j$) in $\C{G}$. 
  Hence $k$ and $j$ are adjacent in $\C{P}$.
  Since this holds for all $k$, the directed edge $i \to j$ in $\C{P}$ cannot be definitely visible. 
\end{proof}
This provides us with a sufficient condition to read off unconfounded pairs of nodes from DPAGs:
\begin{proposition}\label{prop:identify_nonconfounders}
  Let $\C{P}$ be a \typo{DMAG}{DPAG} and $\C{G}$ be a DMG contained in $\C{P}$.
  Let $i \ne j$ be two nodes in $\C{P}$. 
  If $i$ and $j$ are not adjacent in $\C{P}$, or if there is a directed edge $i \to j$ in $\C{P}$ that is definitely visible in $\C{P}$, then $i \oto j$ is absent from $\C{G}$.
\end{proposition}
\begin{proof}
  If $i$ and $j$ are not adjacent in $\C{P}$, then there is no inducing path between $i$ and $j$ in $\C{G}$ by assumption, and in particular, this rules out the presence of the bidirected edge $i \oto j$ in $\C{G}$.

  If $i \to j$ in $\C{P}$ is definitely visible, by Lemma~\ref{lemm:Zhang_invisible}, there cannot be a bidirected edge $i \oto j$ in any DMG contained in $\C{P}$. 
\end{proof}
For example, from the DPAG in Figure~\ref{fig:acyclification} one can infer that there is no bidirected
edge $X_2 \oto X_7$ in the underlying DMG $\C{G}$, as the two nodes are not adjacent in the DPAG,
and also that there is no bidirected edge between $X_2$ and any node in $\{X_3,X_4,X_5,X_6\}$ in $\C{G}$, as all
these edges are definitely visible in the DPAG.

\subsection{IDENTIFYING DIRECT (NON-)CAUSES}

Contrary to DMGs, a directed edge in a DPAG does not necessarily correspond with a \emph{direct} causal relation.
The following proposition provides sufficient conditions to identify the absence of a directed edge from the DPAG.
\begin{proposition}\label{prop:identifiable_direct_noncauses}
  Let $\C{P}$ be a DPAG that contains a DMG $\C{G}$.
  For two nodes $i \ne j$ in $\C{P}$, if $i \ots j$ in $\C{P}$, or $i$ and $j$ are not adjacent in $\C{P}$, then $i \to j$ is not present in $\C{G}$.
\end{proposition}
\begin{proof}
If $i \ots j$ in $\C{P}$, then $i \notin \ansub{\C{G}}{j}$ and hence $i \to j$ cannot be present in $\C{G}$.
If $i$ and $j$ are not adjacent in $\C{P}$, then $i \to j$ cannot be present in $\C{G}$ because this would be an inducing path between $i$ and $j$.
\end{proof}
The following proposition was inspired by Theorem~3 in \citet{borboudakis2012tools} and provides sufficient
conditions to conclude the presence of a directed edge from the DPAG.
\begin{proposition}\label{prop:identifiable_direct_causes}
  Let $\C{P}$ be a DPAG that contains a DMG $\C{G}$.
  For two nodes $i \ne j$ in $\C{P}$, if $i \to j$ in $\C{P}$ and:
  \begin{compactenum}[(i)]
    \item there does not exist a possibly directed path from $i$ to $j$ in $\C{P}$ that avoids the edge $i \to j$, or
    \item if there is no inducing walk between $i$ and $j$ in $\C{G}$ that is both into $i$ and $j$ (for example, because $i \to j$ is definitely visible in $\C{P}$), and for all vertices $k$ such that there is a possibly directed path $i \sts k \sts j$ from $i$ to $j$ in $\C{P}$, the edge $k \to j$ is definitely visible in the DPAG $\C{P}^*$ obtained from $\C{P}$ by replacing the edge between $k$ and $j$ by $k \to j$,
  \end{compactenum}
  then $i \to j$ is present in $\C{G}$.
\end{proposition}
\begin{proof}
  (i) Suppose $i \to j$ in $\C{P}$. We prove the contrapositive. Assume $i \to j$ is absent from $\C{G}$. Because $i \in \ansub{\C{G}}{j}$ by assumption, there must be a directed path from $i$ to $j$ in $\C{G}$ that does not contain the edge $i \to j$. This corresponds with a possibly directed path in $\C{P}$ that avoids the edge $i \to j$.

  (ii) Suppose $i \to j$ in $\C{P}$. 
  There must be an inducing walk $\pi$ between $i$ and $j$ in $\C{G}$ that is into $j$ by Lemma~\ref{lemm:inducing_walk_into}.
  Each collider on $\pi$ is ancestor of $j$ (because it is ancestor of $i$ or $j$ by definition, and $i$ is ancestor of $j$).
  By assumption (or by Lemma~\ref{lemm:Zhang_invisible} if $i \to j$ is definitely visible in $\C{P}$), this inducing walk must be out of $i$.

  We now show that $\pi$ cannot contain any colliders under the assumptions made.
  For the sake of contradiction, assume that $\pi$ contained one or more colliders.
  Denote the collider closest to $i$ on $\pi$ by $k$.
  Since $\pi$ is out of $i$ and $k$ is the collider on $\pi$ closest to $i$, $\pi$ must start as a directed walk $i \to \dots \to k$.
  This is an inducing walk between $i$ and $k$, and since $i \in \ansub{\C{G}}{k}$,
  it corresponds with a possibly directed path $i \sts k$ in $\C{P}$.
  The subwalk of $\pi$ between $k$ and $j$ is an inducing walk in $\C{G}$ between $k$ and $j$ that is both into $k$ and into $j$, since each collider on it is ancestor of $j$ and each non-endpoint non-collider only points to nodes in the same strongly connected component.
  Since $k \in \ansub{\C{G}}{j}$, this corresponds with a possibly directed path $k \sts j$ in $\C{P}$.
  It also implies that $\C{P}^*$, obtained from $\C{P}$ by replacing the edge between $k$ and $j$ by the directed edge $k \to j$, contains $\C{G}$. 
  By Lemma~\ref{lemm:Zhang_invisible} applied to $\C{P}^*$, $k \to j$ cannot be definitely visible in $\C{P}^*$.
  We have arrived at a contradiction with the assumption.

  Hence, the inducing walk $\pi$ cannot contain any colliders.
  If it consisted of multiple edges, it would be of the form $i \to k' \to \dots \to j$, where now $k' \ne j$ is the vertex on $\pi$ next to $i$, and all non-endpoint noncolliders would point to nodes in the same strongly connected component.
  Hence $k'$ and $j$ would lie in the same strongly connected component of $\C{G}$.
  Again, note that this results in a possibly directed path $i \sts k' \sts j$ in $\C{P}$, and means that $\C{P}^*$ obtained from $\C{P}$ by replacing the edge between $k'$ and $j$ by the directed edge $ \to j$, contains $\C{G}$.
  By Lemma~\ref{lemm:Zhang_invisible} applied to $\C{P}^*$, there exists no inducing walk in $\C{G}$ between $k'$ and $j$ that is into $k'$, because $k' \to j$ must be definitely visible in $\C{P}^*$ by assumption.
  This contradicts the existence of a directed path $k' \ot \dots \ot j$ in $\C{G}$.
  
  Hence, the inducing walk $\pi$ in $\C{G}$ must consist of a single edge, and is necessarily of the form $i \to j$.
\end{proof}
As an example, the edge $X_2 \to X_3$ in the DPAG in Figure~\ref{fig:acyclification} cannot be identified
as being present in $\C{G}$ because both conditions are not satisfied: (i) because of the possibly directed path 
$X_2 \to X_4 \ctc X_3$, (ii) because of the same path where the edge $X_4 \to X_3$ would be possibly invisible if oriented
in that way. Also the edge $X_1 \to X_3$ in the DPAG cannot be identified as being present in $\C{G}$. The edge $X_6 \to X_7$ in
the DPAG, on the other hand, is identifiably present in $\C{G}$.

\subsection{IDENTIFIABLE NON-CYCLES}

Strongly connected components in the DMG end up as a specific pattern in the DPAG.
This can be used as a sufficient condition for identifying the absence of certain cyclic causal relations in a complete DPAG.
\begin{proposition}\label{prop:noncycles}
  Let $\C{G}$ be a DMG and denote by $\C{P} = \PAGFCI(\IM_\sigma(\C{G}))$ the corresponding complete DPAG output by FCI.
  Let $i \ne j$ be two nodes in $\C{P}$.
  If $j \in \sccsub{\C{G}}{i}$, then $i \ctc j$ in $\C{P}$, and for all nodes $k$:
  $k \to i$ in $\C{P}$ iff $k \to j$ in $\C{P}$, and $k \oto i$ in $\C{P}$ iff $k \oto j$ in $\C{P}$, and $k \cto i$ in $\C{P}$ iff $k \cto j$ in $\C{P}$.
\end{proposition}
\begin{proof}
Since no pair of nodes within a strongly connected component of $\C{G}$ can be $\sigma$-separated,
each strongly connected component of $\C{G}$ ends up as a fully-connected component
in $\C{P}$. For two nodes $i \ne j$ in the same strongly connected component of $\C{G}$, there 
exists an acyclification of $\C{G}$ in which $i \to j$ and another one in which $i \ot j$, and
hence the edge between $i$ and $j$ in $\C{P}$ must be oriented as $i \ctc j$.
From Lemma~\ref{lemm:Zhang_A1} it then directly follows that for any third node $k$,
$k \sto i \in \C{P}$ if and only if $k \sto j \in \C{P}$.
If $k \oto i \in \C{P}$, then $k \cto j \notin \C{P}$, otherwise $i \oto k \cto j$ would violate Lemma~\ref{lemm:Zhang_A1}.
Also, if $k \oto i \in \C{P}$, then $k \to j \notin \C{P}$, otherwise $k \to j \ctc i$ would violate Lemma~\ref{lemm:Zhang_A1}.

Hence, we have shown that for all $i \ne j$ with $i \in \sccsub{\C{G}}{j}$, $i \ctc j \in \C{P}$ and for all $k$:
\begin{equation}\label{eq:star}
  \begin{cases}
    k \sto i \in \C{P} \iff k \sto j \in \C{P}, \text{ and} & \\
    k \oto i \in \C{P} \iff k \oto j \in \C{P}. &
  \end{cases}\tag{$*$}
\end{equation}
Note that this will already hold for the DPAG $\tilde{\C{P}}$ constructed by the first (arrowhead
orientation) stage of FCI, i.e., after rules $\C{R}$0--$\C{R}$4 of the FCI algorithm (the only ones that can orient
arrow heads) have been completed.
It remains to show that if $k \to i$ in $\C{P}$ for a third node $k$, then $k \to j$ in $\C{P}$ as well (ruling out $k \cto j$).
We will consider all FCI rules that can orient a tail at $k \to i$ in the absence of selection bias,
i.e., FCI rules $\C{R}$1, $\C{R}$4a, $\C{R}$8a, $\C{R}$9, and $\C{R}$10 in \citep{Zhang2008_AI}, and show that each of them implies also a tail at $k$ on the edge to $j$, i.e., $k \to j$ (rules $\C{R}$5--$\C{R}$7 and $\C{R}$8b can be ignored in the absence of selection bias). 
Below we use $\C{P'}$ to denote an intermediate DPAG obtained so far by FCI during the orientation stage, which ultimately results in the completely oriented DPAG $\C{P}$.

We will use the fact that there is a natural ordering in these orientation rules: 
$\C{R}1$ and $\C{R}$4a are part of the arrowhead orientation stage and complete first. 
Then, all instances of $\C{R}9$ can be executed, after which $\C{R}8$a and $\C{R}10$ are 
triggered repeatedly until completion. The latter follows from the
fact that rules $\C{R}$8a (or $\C{R}$8b) and $\C{R}$10 can only orient an
edge $x \cto y$ into $x \to y$, which can never introduce a
new instance that satisfies the pattern of $\C{R}9$ but did not already
satisfy the pattern of $\C{R}$9 before. We will assume (without loss of
generality) that FCI makes use of this particular ordering in the proof below.


Rule $\C{R}$1: \emph{if $m \sto k \cts i$ in $\C{P'}$ and $m$ and $i$ are not adjacent, then orient $m \sto k \to i$}.
Suppose that $k \cts i$ in $\C{P}'$ can been oriented by $\C{R}$1.
By \eref{eq:star}, this means that $k \sto j$ will be in $\tilde{\C{P}}$.
If $k \cto j$ would have remained unoriented in $\tilde{\C{P}}$, then by Lemma~\ref{lemm:Zhang_A1} applied to $m \sto k \cto j$, 
there must be an edge $m \sto j$ in $\tilde{\C{P}}$. 
But then there must also be an edge $m \sto i$ in $\tilde{\C{P}}$, again by Lemma~\ref{lemm:Zhang_A1}. 
This contradicts that $m$ and $i$ are not adjacent in $\C{P}'$.
Hence $k \to j$ must have been oriented in $\tilde{\C{P}}$.

Rule $\C{R}$4a: \emph{if $\pi = \langle x, m_1, \dots, m_n, k, i \rangle$ is a discriminating path for $k$ in $\C{P}'$ and $k \cts i$ is in $\C{P'}$, and if $k \in \mathrm{SepSet}(x,i)$, then orient $k \to i$}.
First note that $j$ cannot be part of the discriminating path $\pi$, as $i \ctc j$ in $\tilde{\C{P}}$. 
By \eref{eq:star}, all nodes $m_s$ and $k$ also have an edge into $j$ in $\tilde{\C{P}}$ with either a tail or circle mark at the other end. 
  So we have $x \sto m_1 \to j$ in $\tilde{\C{P}}$ or $x \sto m_1 \cto j$ in $\tilde{\C{P}}$. 
If $x$ and $j$ were adjacent in $\tilde{\C{P}}$, then the edge between them must be of the form $x \sto j$ (either because of Lemma~\ref{lemm:Zhang_A1} if $m_1 \cto j$, or because of FCI rule $\C{R}$2b if $m_1 \to j$), which would imply that also $x \sto i$ in $\tilde{\C{P}}$ by \eref{eq:star}, contradicting the antecedent of rule $\C{R}$4a.
Now, by induction each edge between $m_s$ and $j$ (for $s=1,\dots,n$) will have been oriented as $m_s \to j$ in $\tilde{\C{P}}$. 
Indeed, first $\C{R}$1 can orient $x \sto m_1 \to j$, which means that $x \sto m_1 \ots m_2 \sto j$ is a discriminating path for $m_2$. 
Suppose $x \sto m_1 (\oto \dots m_{s-1}) \ots m_s \sts j$ is a discriminating path for $m_s$ with $s < n$.
Now $m_s \oto j$ cannot be in $\tilde{\C{P}}$, because if it were, $m_s \oto i$ would also be in $\tilde{\C{P}}$ by \eref{eq:star}, contrary the antecedent of $\C{R}$4a. 
Hence $m_s \in \mathrm{SepSet}(x,j)$, and so the edge between $m_s$ and $j$ can be oriented as $m_s \to j$ by $\C{R}$4a,
which means that $x \sto m_1 (\oto \dots m_s) \ots m_{s+1} \sts j$ must be a discriminating path for $m_{s+1}$.
Hence, $\pi' = \langle x, m_1, \dots, m_n, k, j \rangle$ is also a discriminating path for $k$ in $\tilde{\C{P}}$, and again $k \oto j$ cannot be in $\tilde{\C{P}}$ (otherwise $k \oto i$ in $\tilde{\C{P}}$), and so the edge between $k$ and $j$ can be oriented
by $\C{R}$4a, resulting in $k \to j$ in $\tilde{\C{P}}$.

Rule $\C{R}$9: \emph{if $k \cto i$ in $\C{P}'$, and $\pi = \langle k, m_1, \dots, m_n, i \rangle$ is an uncovered possibly directed path in $\C{P}'$ from $k$ to $i$ such that $m_1$ and $i$ are not adjacent in $\C{P}'$, then orient $k \to i$}.

First note that all nodes on $\pi$ must be ancestor in $\C{G}$ of $i$, and $i$ must be non-ancestor in $\C{G}$ of all other nodes on $\pi$.
Indeed, for any DMAG $\C{H}$ contained in $\C{P}'$, each node $m_l$ must be a non-collider on $\pi$ (all unshielded colliders have already been oriented by rule $\C{R}$0 in $\C{P}'$ by assumption, and there cannot be any on $\pi$ since $\pi$ is possibly directed). 
If $m_l \to m_{l+1}$ in $\C{H}$, then there must be a directed path $m_l \to m_{l+1} \to \dots \to m_n \to i$ in $\C{H}$;
if $m_l \to m_{l-1}$ in $\C{H}$, then there must be a directed path $m_l \to m_{l-1} \to \dots \to m_1 \to k \to i$ in $\C{H}$.
In both cases, $m_l \in \ansub{\C{H}}{i}$, and $i \notin \ansub{\C{H}}{m_l}$.
Since this holds for any DMAG $\C{H}$ contained in $\C{P'}$, it also holds for all DMAGs induced by acyclifications of $\C{G}$.
Hence, by Proposition~\ref{prop:acyclification}, $m_l \in \ansub{\C{G}}{i}$, and $i \notin \ansub{\C{G}}{m_l}$.
This also implies (by the arrowhead completeness of FCI) that there must be an arrowhead on the edge $m_n \sto i$ in $\C{P}'$. 

By \eref{eq:star}, $k \sto j$ will be in $\C{P}$. Both $m_1$ and $m_{n-1}$ are not adjacent to $j$ in $\C{P}$. This follows from the fact that $j$ cannot be ancestor of either of these nodes $m_1, m_{n-1}$ in $\C{G}$, for then $i$ would also be ancestor in $\C{G}$ of that node. Therefore an edge between $m_1$ (or $m_{n-1}$) and $j$ in $\C{P}$ would have to be into $j$ by the arrowhead completeness of FCI, and so by \eref{eq:star} there would also be an edge $m_1 \sto i$ (or $m_{n-1} \sto i$) in $\C{P}$, contrary the antecedent of $\C{R}$9. 
By \eref{eq:star}, also $m_n \sto j$ in $\C{P}'$. 
This edge cannot be $m_n \oto j$, because that would mean that also $m_i \oto i$ in $\C{P}'$ by \eref{eq:star}, a contradiction.
Therefore $\pi' = \langle k, m_1, \dots, m_n, j \rangle$ is also an uncovered possibly directed path from $k$ to $j$ in $\C{P}'$, and $m_1$ and $j$ are not adjacent in $\C{P}'$, so $k \cto j$ can be oriented as $k \to j$ by $\C{R}$9.

Finally, the two remaining rules, $\C{R}$8a and $\C{R}$10, will be considered together in order to be able to make use of a proof by induction. We will assume that in $\C{P}'$, rules $\C{R}$1, $\C{R}$4a and $\C{R}$9 have already been exhaustively applied.

Rule $\C{R}$8a: \emph{if $k \cto i$ and $k \to m \to i$ in $\C{P'}$, then orient $k \to i$.}\\
Rule $\C{R}$10: \emph{if $k \cto i$ and $u_1 \to i \ot u_2$ in $\C{P'}$, $\pi_1 = \langle k, m_1, \dots, u_1 \rangle$ is an uncovered possibly directed path from $k$ to $u_1$ in $\C{P}'$ and $\pi_2 = \langle k, m_2, \dots, u_2 \rangle$ is an uncovered possibly directed path from $k$ to $u_2$ in $\C{P}'$, such that $m_1$ and $m_2$ are distinct and not adjacent, then orient $k \to i$.}

By \eref{eq:star}, if rule $\C{R}$8a triggers in $\C{P}'$, then also $k \sto j$ and $m \sto j$ in $\C{P}$.
Furthermore, if rule $\C{R}$10 triggers in $\C{P}'$, then also $k \sto j$, $u_1 \sto j$ and $u_2 \sto j$ in $\C{P}$.
As the arrowhead stage has already completed by the time $\C{R}$8a and $\C{R}$10 are executed, that means these will be present as edges into $j$ in $\C{P'}$ as well. 

We now proceed by contradiction. 
Assume, for the sake of contradiction, that $\C{R}$8a or $\C{R}$10 triggers to orient some $k \cto i$ in $\C{P}'$ as $k \to i$, but the corresponding edge $k \cto j$ (with $j \in \sccsub{\C{G}}{i}$) remains unoriented in $\C{P}$.
Consider the first edge $k \to i$ for which this situation occurs (in the sequence of orientations performed by FCI during this last part of the orientation phase).

If $k \cto i$ can be oriented by rule $\C{R}$8a, then the edge $m \to i$ is already present in $\C{P}'$ at that point. 
If the tail on that edge was oriented by one of the rules $\C{R}$1, $\C{R}$4a or $\C{R}$9, also $m \to j$ will have been oriented in $\C{P}'$ at this point, as we have shown. 
Otherwise, the tail of $m \to i$ must have been oriented by rule $\C{R}$8a or rule $\C{R}$10. 
By assumption, the corresponding edge $m \cto j$ does not remain unoriented in $\C{P}$, and therefore after finitely many applications of rules $\C{R}$8a and $\C{R}$10, it will have been oriented as $m \to j$ in $\C{P}'$. 
In both cases, the edge $m \to j$ will be present in $\C{P}'$ at some point, and then rule $\C{R}$8a can also be used to orient $k \cto j$ as $k \to j$, contradicting the assumption that $k \cto j$ remains unoriented in $\C{P}$. 

Hence $k \cto i$ must have been oriented by rule $\C{R}$10. 
The edges $u_1 \to i$ and $u_2 \to i$ must then already be present in $\C{P}'$ at that point.
By similar reasoning as before, we conclude that also $u_1 \to i$ and $u_2 \to i$ must be present in $\C{P}'$ at some point, and then rule $\C{R}$10 can also be used to orient $k \cto j$ as $k \to j$, again contradicting the assumption.

Therefore, it cannot happen that $\C{R}$8a or $\C{R}$10 can be used to orient $k \cto i$ as $k \to i$, but that the corresponding edge $k \cto j$ (with $j \in \sccsub{\C{G}}{i}$) will not get oriented as $k \to j$ in $\C{P}$.

Summarizing, whenever FCI orients a tail mark at $k$ on an edge into $i$, it will also orient all tail marks at $k$ on the edges into $j$ for all $j \in \sccsub{\C{G}}{i}$.
\end{proof}
Hence, any pair of nodes that does not fit this pattern cannot be part of a cycle in $\C{G}$.
For example, in the complete DPAG in Figure~\ref{fig:acyclification}, only the nodes in $\{X_3,X_4,X_5,X_6\}$ might
be part of a cycle. For all other pairs of nodes, it follows from Proposition~\ref{prop:noncycles} that
they cannot be part of a cycle.
This sufficient condition is also necessary:
\begin{proposition}\label{prop:noncycles2b}
  Let $\C{G}$ be a DMG and denote by $\C{P} = \PAGFCI(\IM_\sigma(\C{G}))$ the corresponding complete DPAG output by FCI.
  Let $i \ne j$ be two nodes in $\C{P}$.
  If there is an edge $i \ctc j$ in $\C{P}$, and all nodes $k$ for which $k \sto i$ is in $\C{P}$ also have an edge of the same type $k \sto j$ (i.e., the two edge marks at $k$ are the same) in $\C{P}$, then there exists a DMG $\tilde{\C{G}}$ with $j \in \sccsub{\tilde{\C{G}}}{i}$ that is $\sigma$-Markov equivalent to $\C{G}$, but also \typo{a DMG}{an ADMG} $\C{H}$ with $j \notin \sccsub{\C{H}}{i}$ that is $\sigma$-Markov equivalent to $\C{G}$. 
\end{proposition}
\begin{proof}
We first show that there is a DMAG $\C{H}$ that is $\sigma$-Markov equivalent to $\C{G}$ which has $i \to j$.
We construct $\C{H}$ by starting from the so-called arrowhead augmented graph \citep{Zhang2006} of $\C{P}$, in which each edge $x \cto y$ in $\C{P}$ is oriented as $x \to y$, followed by an orientation of the remaining circle component into an arbitrary DAG such that no unshielded colliders are introduced. 
By Lemma 4.3.6 in \citep{Zhang2006}, this procedure yields a DMAG $\C{H}$ in $\C{P}$ (note that we assumed no selection bias), and hence $\sigma$-Markov equivalent to $\C{G}$.
Orienting the circle component (which is chordal by Lemma 4.3.7 in \citep{Zhang2006}) can be achieved by choosing $i \to j$ as the top two root nodes in the partial ordering for the DAG, and propagating the Meek rules \citep{Meek1995PC} to orient the rest of the circle component into a DAG with no unshielded colliders (see also the discussion of ``Meek's Algorithm'' on pages 120--121 of \citep{Zhang2006}). 
The resulting DMAG $\C{H}$ has no other arrowheads at $i$ and $j$ other than the ones already present in $\C{P}$, in combination with the directed edge $i \to j$.

We now create a DMG $\tilde{\C{G}}$ out of $\C{H}$ by adding an additional edge $j \to i$, thereby creating a single non-trivial strongly connected component containing only the nodes $\{i,j\}$.
This follows from the fact that if there exists another node in $\sccsub{\tilde{\C{G}}}{i}$ as a result of adding the edge $j \to i$ to $\C{H}$, then there must now be a directed path $j \to i \to \dots \to k \to j$ in $\tilde{\C{G}}$. However, this edge $k \to j$ cannot be part of the original circle component containing $i \ctc j$ in $\C{P}$, for then by construction it would have been oriented as $j \to k$, and if $k \to j$ was already an invariant arrowhead in $\C{P}$ then by assumption there would also be an invariant arrowhead on $k \sto i$ in $\C{P}$, which would imply the existence of an (almost) directed cycle $i \to \dots \to k$ in combination with $k \sto i$ in the DMAG $\C{H}$, which is impossible by definition.

As final step in the proof we show that the DMAG $\C{H}$ qualifies as an acyclification of $\tilde{\C{G}}$.
By construction $\C{H}$ is a DMAG, and also an ADMG, with the same vertices as $\tilde{\C{G}}$. 
Since $j \in \sccsub{\tilde{\C{G}}}{i}$ and $i \to j$ is in $\C{H}$, criterion (\ref{def:acyclification_intra_scc}) of Definition~\ref{def:acyclification} is satisfied. 
Finally, note that for all other nodes $k \notin \sccsub{\tilde{\C{G}}}{i}$ by assumption, if $k \to i$ in $\C{P}$, then $k \to j$ is also in $\C{P}$, and similarly for $k \oto i$ and $k \oto j$, resp.\ $k \cto i$ and $k \cto j$. 
By the initial step of the construction of $\C{H}$ as the arrowhead augmented DMAG of $\C{P}$, any edge $k \cto i$ in $\C{P}$ is oriented as $k \to i$ in $\C{H}$. 
As a result, for any node $k \notin \sccsub{\tilde{\C{G}}}{i}$, $k \to i$ in $\C{H}$ iff $k \to j$ in $\C{H}$, and similarly, $k \oto i$ in $\C{H}$ iff $k \oto j$ in $\C{H}$. 
Since $\{i,j\}$ is the only non-trivial strongly connected component in $\tilde{\C{G}}$, criterion (\ref{def:acyclification_inter_scc}) of Definition~\ref{def:acyclification} is also satisfied. 
This proves that DMAG $\C{H}$ indeed qualifies as an acyclification of $\tilde{\C{G}}$. 

This means that $\tilde{\C{G}}$ is $\sigma$-Markov equivalent to $\C{G}$, and hence, both $\tilde{\C{G}}$ and $\C{H}$ satisfy all properties promised in the claim of the proposition. 
\end{proof}
In other words, under the conditions of this proposition, it is not identifiable from $\C{P}$ alone whether $j$ and $i$ are part of a causal cycle.

\section{EXTENSIONS FOR BACKGROUND KNOWLEDGE}

In this section, we discuss extensions of our results to situations in which available causal background knowledge is taken into account by causal discovery algorithms.

Assume that we have certain background knowledge, formalized as a Boolean function $\Psi$ on the set of all DMGs (indicating for each DMG whether it satisfies the background knowledge).
For example, one type of background knowledge commonly considered in the literature (probably mainly for reasons of simplicity) is \emph{causal sufficiency}, which can be formalized by $\Psi(\C{G}) = 1$ iff $\C{G}$ contains no bidirected edges, and $\Psi(\C{G}) = 0$ otherwise. 
A less trivial example of background knowledge are the JCI Assumptions, which play a central role in the Joint Causal Inference framework \citep{Mooij++_JMLR_2020} for performing causal discovery from multiple datasets that correspond with measurements of a system in different contexts (for example, a combination of observational and different interventional datasets). 
The latter example will be discussed in more detail in Section~\ref{sec:jci}.

\subsection{SOUNDNESS AND COMPLETENESS}

We first extend the standard notions of soundness and completeness to a setting that involves cycles and background knowledge (but no selection bias).
\begin{definition}\label{def:sound_complete_background_knowledge}
  Under background knowledge $\Psi$, a mapping $\Phi$ from independence models to DPAGs is called:
  \begin{compactitem}
    \item \emph{sound} if for all DMGs $\C{G}$ with $\Psi(\C{G}) = 1$: $\Phi(\IM_\sigma(\C{G}))$ contains $\C{G}$; 
    \item \emph{arrowhead complete} if for all DMGs $\C{G}$ with $\Psi(\C{G}) = 1$: if $i \notin \ansub{\tilde{\C{G}}}{j}$ for any DMG $\tilde{\C{G}}$ with $\Psi(\tilde{\C{G}}) = 1$ that is $\sigma$-Markov equivalent to $\C{G}$, then there is an arrowhead $i \ots j$ in $\Phi(\IM_\sigma(\C{G}))$;
    \item \emph{tail complete} if for all DMGs $\C{G}$ with $\Psi(\C{G}) = 1$: if $i \in \ansub{\tilde{\C{G}}}{j}$ in any DMG $\tilde{\C{G}}$ with $\Psi(\tilde{\C{G}}) = 1$ that is $\sigma$-Markov equivalent to $\C{G}$, then there is a tail $i \to j$ in $\Phi(\IM_\sigma(\C{G}))$;
    \item \emph{Markov complete} if for all DMGs $\C{G}_1,\C{G}_2$ with $\Psi(\C{G}_1)=\Psi(\C{G}_2)=1$: $\C{G}_1$ is $\sigma$-Markov equivalent to $\C{G}_2$ iff $\Phi(\IM_\sigma(\C{G}_1)) = \Phi(\IM_\sigma(\C{G}_2))$.
  \end{compactitem}
  It is called complete if it is both arrowhead complete and tail complete.
\end{definition}
Note that this reduces to the standard notions \citep{Zhang2008_AI} if $\Psi(\C{G}) = 1$ iff $\C{G}$ is acyclic,
while it also reduces to the notions in Theorem~\ref{theo:fci_sound_complete} if no background knowledge is used
(i.e., $\Psi(\C{G}) = 1$ for all $\C{G}$).

We assume that the background knowledge is \emph{compatible with the acyclification} in the following sense:
\begin{assumption}\label{ass:acybk}
  For all DMGs $\C{G}$ with $\Psi(\C{G}) = 1$, the following three conditions hold:
  \begin{compactenum}[(i)]
    \item There exists an acyclification $\C{G}'$ of $\C{G}$ with $\Psi(\C{G}') = 1$;\label{ass:acybk_nonempty}
    \item For all nodes $i,j$ in $\C{G}$: if $i \in \ansub{\C{G}}{j}$ then there exists an acyclification $\C{G}'$ of $\C{G}$ with $\Psi(\C{G}') = 1$ such that $i \in \ansub{\C{G}'}{j}$;\label{ass:acybk_anc}
    \item For all nodes $i,j$ in $\C{G}$: if $i \notin \ansub{\C{G}}{j}$ then $i \notin \ansub{\C{G}'}{j}$ for all acyclifications $\C{G}'$ of $\C{G}$ with $\Psi(\C{G}') = 1$.\label{ass:acybk_nonanc}
  \end{compactenum}
\end{assumption}
For example, the background knowledge of ``causal sufficiency'' satisfies this assumption, as well as the background knowledge of ``acyclicity''.

The following result is straightforward given all the definitions, but is also quite powerful, as it allows us to directly generalize existing acyclic soundness and completeness results (for certain background knowledge) to the $\sigma$-separation setting.
\begin{theorem}\label{theo:generalizing_soundness_completeness}
  Let $\Psi$ be background knowledge that satisfies Assumption~\ref{ass:acybk} and let $\Phi$ be a mapping from independence models to DPAGs. Then:
  \begin{compactenum}[(i)]
    \item If $\Phi$ is sound for background knowledge $\Psi$ under the additional assumption of acyclicity, then $\Phi$ is sound for background knowledge $\Psi$.
    \item If $\Phi$ is arrowhead (tail) complete for background knowledge $\Psi$ under the additional assumption of acyclicity, then $\Phi$ is arrowhead (tail) complete for background knowledge $\Psi$.
     \item If $\Phi$ is sound and arrowhead complete for background knowledge $\Psi$ under the additional assumption of acyclicity, then $\Phi$ is Markov complete.
  \end{compactenum}
\end{theorem}
\begin{proof}
For a DMG $\C{G}$ with $\Psi(\C{G}) = 1$, define 
  \begin{equation*}
    \begin{split}
      \acy(\C{G},\Psi) := \{ \C{G}' : {} & \C{G}' \text{ is an acyclification of $\C{G}$} \\
                                         & \text{ and $\Psi(\C{G}') = 1$}\}.
    \end{split}
  \end{equation*}

  (i) Suppose that $\Phi$ is sound for background knowledge $\Psi$ under the additional assumption of acyclicity.
  Let $\C{G}$ be a DMG with $\Psi(\C{G}) = 1$ and let $\C{P} := \Phi(\IM_\sigma(\C{G}))$.
  For any $\C{G}' \in \acy(\C{G},\Psi)$, which is nonempty by Assumption~\ref{ass:acybk}(\ref{ass:acybk_nonempty}), $\C{P} = \Phi(\IM_\sigma(\C{G}'))$ contains $\C{G}'$ by virtue of the acyclic soundness of $\Phi$ under background knowledge $\Psi$. 
  Hence, two nodes $u,v$ in $\C{P}$ are adjacent if and only if there is an inducing path between $u$ and $v$ in $\C{G}'$, and by  Proposition~\ref{prop:acyclification}(\ref{prop:acyclification_inducing_path}), this holds if and only if there is an inducing path between $u$ and $v$ in $\C{G}$.
  Further, $u \sto v$ in $\C{P}$ implies $v \notin \ansub{\C{G}'}{u}$ for all $\C{G}' \in \acy(\C{G},\Psi)$, and hence by Assumption~\ref{ass:acybk}(\ref{ass:acybk_anc}), this implies $v \notin \ansub{\C{G}}{u}$.
  Finally, $u \to v$ in $\C{P}$ implies $u \in \ansub{\C{G}'}{v}$ for all $\C{G}' \in \acy(\C{G},\Psi)$, and hence by Assumption~\ref{ass:acybk}(\ref{ass:acybk_nonanc}), this implies $u \in \ansub{\C{G}}{v}$.
  We conclude that $\Phi$ is sound for background knowledge $\Psi$.

  (ii) We give the proof for arrowhead completeness (tail completeness is proved similarly).
  Suppose that $\Phi$ is arrowhead complete for background knowledge $\Psi$ under the additional assumption of acyclicity.
  I.e., for all ADMGs $\C{G}'$ with $\Psi(\C{G}') = 1$, any ancestral relation absent in any ADMG $\tilde{\C{G}}'$ with $\Psi(\tilde{\C{G}}') = 1$ that is $d$-Markov equivalent to $\C{G}'$, is oriented as an arrowhead in $\Phi(\IM_d(\C{G}'))$.
  Let $\C{G}$ be a DMG with $\Psi(\C{G}) = 1$ and let $\C{P} := \Phi(\IM_\sigma(\C{G}))$.
  Let $\C{G}' \in \acy(\C{G},\Psi)$, which is nonempty by Assumption~\ref{ass:acybk}(\ref{ass:acybk_nonempty}).
  Assume that an ancestral relation is absent in any DMG $\tilde{\C{G}}$ with $\Psi(\tilde{\C{G}}) = 1$ that is $\sigma$-Markov equivalent to $\C{G}$. 
  Then in particular, this ancestral relation is absent in any ADMG $\tilde{\C{G}}'$ with $\Psi(\tilde{\C{G}}') = 1$ that is $d$-Markov equivalent to $\C{G}'$ (which is $\sigma$-Markov equivalent to $\C{G}$).
  By the acyclic arrowhead completeness of $\Phi$ under background knowledge $\Psi$, it must be oriented as an arrowhead in $\Phi(\IM_d(\C{G}')) = \C{P}$.
  Hence $\Phi$ is arrowhead complete for background knowledge $\Psi$.

  (iii) Let $\C{G}_1$ and $\C{G}_2$ be two DMGs that satisfy the background knowledge, i.e., $\Psi(\C{G}_1) = \Psi(\C{G}_2) = 1$.
We have to show that if $\C{G}_1$ and $\C{G}_2$ are not $\sigma$-Markov equivalent, then $\Phi(\IM_\sigma(\C{G}_1)) \ne \Phi(\IM_\sigma(\C{G}_2))$.
  By Assumption~\ref{ass:acybk}(\ref{ass:acybk_nonempty}), there exist acyclifications $\C{G}_1' \in \acy(\C{G}_1,\Psi)$
and $\C{G}_2' \in \acy(\C{G}_2,\Psi)$.
Proposition~\ref{prop:Markov_equivalence_acyclification} implies both $\IM_\sigma(\C{G}_1) = \IM_d(\C{G}_1')$ and $\IM_\sigma(\C{G}_2) = \IM_d(\C{G}_2')$. 
Assume that $\C{G}_1$ and $\C{G}_2$ are not $\sigma$-Markov equivalent. 
Then their acyclifications $\C{G}_1'$ and $\C{G}_2'$ are not $d$-Markov equivalent.
By assumption, both acyclifications $\C{G}_1', \C{G}_2'$ satisfy the background knowledge, i.e., $\Psi(\C{G}_1') = \Psi(\C{G}_2') = 1$.

The induced DMAGs $\C{H}_1 := \DMAG(\C{G}_1')$ and $\C{H}_2 := \DMAG(\C{G}_2')$ cannot be $d$-Markov equivalent because the induced DMAGs preserve the conditional independence models.
Therefore, by the result of \citet{Ali++2009}, $\C{H}_1$ and $\C{H}_2$ either have a different skeleton, or they have the same skeleton but different colliders with order.
If their skeletons differ, then also $\Phi(\IM_d(\C{G}_1')) \ne \Phi(\IM_d(\C{G}_2'))$ by the assumed soundness of $\Phi$. 
If they do have the same skeletons, there must be at least one collider with order in $\C{H}_1$ that is not a collider with order in $\C{H}_2$, or vice versa.
By Lemma 3.13 in \citet{Ali++2009}, this would imply that there is at least one collider with order in $\C{H}_1$ that is not a collider in $\C{H}_2$, or vice versa. 

Without loss of generality, assume that the former holds.  
The assumed soundness and arrowhead completeness of $\Phi$ under the additional assumption of acyclicity imply that this collider with order in $\C{H}_1 = \DMAG(\C{G}_1')$ will appear as a collider in $\Phi(\IM_d(\C{G}_1'))$.
Also, the soundness of $\Phi$ under the additional assumption of acyclicity implies that this noncollider in $\C{H}_2$ cannot end up as a collider in $\Phi(\IM_d(\C{G}_2'))$. 
Hence, $\Phi(\IM_d(\C{G}_1')) \ne \Phi(\IM_d(\C{G}_2'))$, and therefore, $\Phi(\IM_\sigma(\C{G}_1)) \ne \Phi(\IM_\sigma(\C{G}_2))$.
\end{proof}

In the remainder of this section, we will apply this result to two types of background knowledge: causal sufficiency, and the JCI assumptions.

\subsection{CAUSAL SUFFICIENCY}\label{sec:causal_sufficiency}

We consider the (commonly assumed) background knowledge of ``causal sufficiency''. 
This is formalized by $\Psi(\C{G}) = 1$ iff DMG $\C{G}$ contains no bidirected edges.
For the acyclic setting, the well-known PC algorithm \citep{SGS2000}, adapted with Meek's orientation rules \citep{Meek1995PC}, was shown to be sound and complete.
It outputs a so-called Complete Partially Directed Acyclic Graph (CPDAG), which can be interpreted also as a DPAG (by replacing all undirected edges $i \ttt j$ by bicircle edges $i \ctc j$).
Because this particular background knowledge satisfies Assumption~\ref{ass:acybk}, we can apply Theorem~\ref{theo:generalizing_soundness_completeness} to extend the existing acyclic soundness and completeness results to the cyclic setting:
\begin{corollary}
  The PC algorithm with Meek's orientation rules is sound, arrowhead complete, tail complete and Markov complete (in the $\sigma$-separation setting without selection bias).
\end{corollary}
We can therefore also apply Propositions \ref{prop:cdpag_ancestors_dmg}, \ref{prop:cdpag_nonancestors},
to read off the absence or presence of indirect 
causal relations 
from the DPAG (obtained from the CPDAG) output by the PC algorithm.
Note that the presence or absence of direct causal relations can be easily read off from the DPAG in this case as they are in one-to-one correspondence with directed edges in the DPAG.

\subsection{JOINT CAUSAL INFERENCE}\label{sec:jci}
Recently, \citet{Mooij++_JMLR_2020} proposed FCI-JCI, an extension of FCI that enables causal discovery from data measured in different contexts (for example, if observational data as well as data corresponding to various interventions is available).
This is a particular implementation of the general Joint Causal Inference (JCI) framework.
For a detailed treatment, we refer the reader to \citep{Mooij++_JMLR_2020}; here we only give a brief summary of the JCI assumptions that we need to extend our results on FCI to FCI-JCI.
\begin{definition}[JCI Assumptions]\label{def:jci_assumptions}
  The data-generating mechanism for a system in a context is described by a simple SCM $\C{M}$ with two types of 
  endogenous variables: \emph{system} variables $\{X_i\}_{i\in\C{I}}$ and \emph{context} variables $\{C_k\}_{k\in\C{K}}$. 
  Its graph $\C{G}(\C{M})$ has nodes $\C{I} \cup \C{K}$ (corresponding to system variables and context variables, respectively).
  The following (optional) JCI Assumptions can be made about the graph $\C{G} := \C{G}(\C{M})$:
  \begin{compactenum}[(1)]
    \item \label{ass:uncaused}
      \emph{Exogeneity}: No system variable causes any context variable, i.e.,
      $\forall_{k \in \C{K}} \forall_{i\in\C{I}}: i \to k \notin \C{G}$.
    \item \label{ass:unconfounded}
      \emph{Randomization}: No pair of context and system variable is confounded, i.e.,
      $\forall_{k\in\C{K}} \forall_{i\in\C{I}}: i \oto k \notin \C{G}$.
    \item \label{ass:dependences}
      \emph{Genericity}: The induced subgraph $\C{G}(\C{M})_{\C{K}}$ on the context variables is of the following special form: 
      $\forall_{k \ne k' \in \C{K}}: k \oto k' \in \C{G} \land k \to k' \notin \C{G}$.
  \end{compactenum}
\end{definition}
The following Lemma is key to our extensions to the cyclic $\sigma$-separation setting.
\begin{lemma}\label{lemm:jci_acyclification}
  If subset $\{1\}$, $\{1,2\}$, or $\{1,2,3\}$ of the JCI Assumptions holds for a DMG $\C{G}$, then the same subset of assumptions holds for any acyclification of $\C{G}$.
\end{lemma}
\begin{proof}
  Let $\C{G}'$ be an acyclification of $\C{G}$.
  JCI Assumption 1 implies that each strongly connected component in $\C{G}$ consists entirely of system variables
  or entirely of context variables. 
  Since in addition, $\C{G}$ does not have any directed edge from a system to a context variable,
  there will not be any spurious directed edge in $\C{G}'$ from a system to a context variable. 
  Hence also $\C{G}'$ satisfies JCI Assumption 1. 
  If $\C{G}$ satisfies also JCI Assumption 2, the acyclification $\C{G}'$ will not contain any spurious bidirected edge between a context and a system variable. 
  Hence $\C{G}'$ satisfies JCI Assumptions 1 and 2 if $\C{G}$ does so. 
  Finally, it is clear that JCI Assumption 3 holds for $\C{G}'$ if JCI Assumptions 1 and 3 hold for $\C{G}$.
\end{proof}
This trivially implies that these different combinations of the JCI Assumptions satisfy 
Assumption~\ref{ass:acybk}. That allows us to extend the existing acyclic soundness and
completeness results for FCI-JCI to the cyclic setting.

FCI-JCI was shown to be sound under the assumption of acyclicity \citep[Theorem 35,][]{Mooij++_JMLR_2020}.
This gives with Theorem~\ref{theo:generalizing_soundness_completeness}:
\begin{corollary}\label{coro:fcijci_sound}
  For the background knowledge consisting of JCI Assumptions $\emptyset$, $\{1\}$, $\{1,2\}$ or $\{1,2,3\}$,
  the corresponding version of FCI-JCI is sound (in the $\sigma$-separation setting without selection bias).
\end{corollary}
We can therefore also apply Propositions \ref{prop:cdpag_nonancestors} and \ref{prop:identify_nonconfounders} to
read off the absence of indirect causal relations and confounding from the DPAG output by the FCI-JCI algorithm,
and Propositions \ref{prop:identifiable_direct_noncauses} and \ref{prop:identifiable_direct_causes} to read off the absence or presence of direct causal relations.
Furthermore, it is clear from its definition that all unshielded triples in the DPAG that FCI-JCI outputs have been oriented according to FCI rule $\C{R}0$.
Therefore, we can also apply Proposition \ref{prop:cdpag_ancestors_dmg} to read off the presence of indirect causal relations from the DPAG output by the FCI-JCI algorithm.

Under all three JCI assumptions, stronger results have been derived. 
In particular, completeness of FCI-JCI has been shown \citep[Theorem 38][]{Mooij++_JMLR_2020} under the background knowledge of all three JCI Assumptions in the acyclic setting.
This gives with Theorem~\ref{theo:generalizing_soundness_completeness}:
\begin{corollary}\label{coro:fcijci123_complete}
  For the background knowledge consisting of JCI Assumptions $\{1,2,3\}$, the FCI-JCI algorithm is 
  arrowhead complete, tail complete and Markov complete (in the $\sigma$-separation setting without selection bias).
\end{corollary}


An important feature of Joint Causal Inference under JCI Assumptions $\{1,2,3\}$ is that the direct (non-)targets of interventions need not be known, but can be discovered from the data. 
The sufficient condition provided in Proposition 42 of \citet{Mooij++_JMLR_2020} can be easily generalized to the $\sigma$-separation setting
as well by observing that under JCI Assumptions $\{1,2,3\}$,
there cannot be an inducing walk between a system node and a context node that is into both,
and then applying Proposition~\ref{prop:identifiable_direct_noncauses} and
Proposition~\ref{prop:identifiable_direct_causes}.
\begin{proposition}\label{prop:fci_jci123_direct_intervention_targets}
Let $\C{G}$ be a DMG that satisfies JCI Assumptions $\{1,2,3\}$.
Let $\C{P} = \PAGFCIJCI(\IM_\sigma(\C{G}))$ denote the DPAG output by the corresponding version of FCI-JCI. 
Let $i \in \C{K}$, $j \in \C{I}$. Then:
  \begin{compactenum}
    \item If $i$ is not adjacent to $j$ in $\C{P}$, $i \to j$ is not in $\C{G}$.
    \item If $i \to j$ in $\C{P}$, and 
      for all system nodes $k \in \C{I}$ s.t.\ $i \to k$ in $\C{P}$ and $k \ctc j$ or $k \cto j$ or $k \to j$ in $\C{P}$, the edge $k \to j$ is definitely visible in the DPAG $\C{P}^*$ obtained from $\C{P}$ by replacing the edge between $k$ and $j$ by $k \to j$,
    then $i \to j$ is present in $\C{G}$.
  \end{compactenum}
\end{proposition}
\begin{proof}
  By Corollary~\ref{coro:fcijci_sound}, $\C{P}$ contains $\C{G}$.
  The first statement then follows directly from Proposition~\ref{prop:identifiable_direct_noncauses}.

  For the second statement, note first that there must be an inducing walk $\pi$ between $i$ and $j$ in $\C{G}$ that is into $j$ by Lemma~\ref{lemm:inducing_walk_into}.
  We will show that any such inducing walk cannot be into $i$.
  On the contrary, suppose that $\pi$ would be into $i$.
  Because of JCI Assumptions 1 and 2, the node on $\pi$ next to $i$ cannot be in $\C{I}$ but must be a context node in $\C{K}$.
  Similarly, all subsequent nodes on the inducing \typo{path}{walk} (except for the final node $j$) must be collider nodes in $\C{K}$ because of JCI Assumptions 1, 2 and 3.  
  But then the final edge of $\pi$ is between a context node and system node $i$ and into the context node, contradicting JCI Assumption 1 or 2.

  The claim now follows from the second statement of Proposition~\ref{prop:identifiable_direct_causes}.
\end{proof}

Furthermore, also Proposition~\ref{prop:noncycles} that allows one to identify the absence
of cycles can be extended to FCI-JCI under JCI Assumptions $\{1,2,3\}$.
\begin{proposition}\label{prop:fci_jci123_noncycles}
  Let $\C{G}$ be a DMG that satisfies JCI Assumptions $\{1,2,3\}$.
  Let $\C{P} = \PAGFCIJCI(\IM_\sigma(\C{G}))$ denote the complete DPAG output by the corresponding version of FCI-JCI.
  Let $i \ne j$ be two nodes in $\C{P}$.
  If $j \in \sccsub{\C{G}}{i}$, then $i \ctc j$ in $\C{P}$, and for all nodes $k \ne i,j$:
  $k \to i$ in $\C{P}$ iff $k \to j$ in $\C{P}$, and $k \oto i$ in $\C{P}$ iff $k \oto j$ in $\C{P}$, and $k \cto i$ in $\C{P}$ iff $k \cto j$ in $\C{P}$.
\end{proposition}
\begin{proof}
  We make use of a similar strategy as in the proof of \citep[Theorem 38 in][]{Mooij++_JMLR_2020}, which we will not repeat here. Proposition~\ref{prop:noncycles} applies to the extended complete DPAG $\C{P}^*$ constructed in that proof. Since strongly connected components can only occur amongst the system variables, the same orientations will be found in the DPAG $\C{P}$ output by FCI-JCI.
\end{proof}

\section{DISCUSSION AND CONCLUSION}

We have shown that, surprisingly, the FCI algorithm and several of its variants that were designed for
the acyclic setting need not be adapted but directly apply also in the cyclic setting
under the assumptions of the $\sigma$-Markov property, $\sigma$-faithfulness, and the
absence of selection bias. Furthermore, we have provided sufficient conditions
to identify causal features from the DPAG output by FCI and its variants.
For convenience, we state this as a corollary, collecting several of our results.
\begin{corollary}
  Let $\C{M}$ be a simple (possibly cyclic) SCM with graph $\C{G}(\C{M})$ and assume that its distribution
  $\Prb_{\C{M}}(\B{X})$ is $\sigma$-faithful w.r.t.\ the graph $\C{G}(\C{M})$.
  When using consistent conditional independence tests on an i.i.d.\ sample of observational data from the induced distribution $\Prb_{\C{M}}(\B{X})$ of $\C{M}$, FCI provides a consistent estimate $\hat{\C{P}}$ of the DPAG $\PAGFCI(\IM_\sigma(\C{G}(\C{M})))$ that represents the $\sigma$-Markov equivalence class of $\C{G}(\C{M})$. From the estimated DPAG $\hat{\C{P}}$, we obtain consistent estimates for:
   (i) the absence/presence of (possibly indirect) causal relations according to $\C{M}$ via Propositions~\ref{prop:cdpag_ancestors_dmg} and \ref{prop:cdpag_nonancestors}; 
   (ii) the absence of confounders according to $\C{M}$ via Proposition~\ref{prop:identify_nonconfounders};
   (iii) the absence/presence of direct causal relations according to $\C{M}$ via Propositions~\ref{prop:identifiable_direct_noncauses} and \ref{prop:identifiable_direct_causes};
   (iv) the absence of causal cycles according to $\C{M}$ via Proposition~\ref{prop:noncycles}.
\end{corollary}
\begin{proof}
In general, soundness of a constraint-based causal discovery algorithm (i.e.,
the correctness of its output when given the true independence model as input)
implies consistency of the algorithm when using appropriate conditional independence tests.
\end{proof}
A similar conclusion can be formulated for the FCI-JCI algorithm:
\begin{corollary}
  Let $\C{M}$ be a simple (possibly cyclic) SCM with graph $\C{G}(\C{M})$ that satisfies JCI Assumptions $\emptyset$, $\{1\}$, $\{1,2\}$ or $\{1,2,3\}$ and assume that its distribution $\Prb_{\C{M}}(\B{X},\B{C})$ is $\sigma$-faithful w.r.t.\ the graph $\C{G}(\C{M})$. 
  When using consistent conditional independence tests on an i.i.d.\ sample of data from the induced joint distribution $\Prb_{\C{M}}(\B{X},\B{C})$ of $\C{M}$, FCI-JCI provides a consistent estimate $\hat{\C{P}}$ of the DPAG $\C{P} := \PAGFCIJCI(\IM_\sigma(\C{G}(\C{M})))$, which contains $\C{G}$.
  From the estimated DPAG $\hat{\C{P}}$, we obtain consistent esimates for:
   (i) the absence/presence of (possibly indirect) causal relations according to $\C{M}$ via Propositions~\ref{prop:cdpag_ancestors_dmg} and \ref{prop:cdpag_nonancestors}; 
   (ii) the absence of confounders according to $\C{M}$ via Proposition~\ref{prop:identify_nonconfounders};
   (iii) the absence/presence of direct causal relations according to $\C{M}$ via Proposition~\ref{prop:identifiable_direct_causes}.

  In the special case that $\C{G}(\C{M})$ satisfies all three JCI Assumptions $\{1,2,3\}$, the DPAG $\hat{\C{P}}$ estimated by FCI-JCI also provides consistent estimates for:
  (iv) the direct intervention targets and non-targets according to $\C{M}$ via Proposition~\ref{prop:fci_jci123_direct_intervention_targets};
  (v) the absence of causal cycles according to $\C{M}$ via Proposition~\ref{prop:fci_jci123_noncycles}.
  Furthermore, the DPAG $\C{P} := \PAGFCIJCI(\IM_\sigma(\C{G}(\C{M})))$ characterizes the DMGs that satisfy JCI Assumptions $\{1,2,3\}$ and are $\sigma$-Markov equivalent to $\C{G}$.
\end{corollary}
Obviously, our results apply also in the acyclic setting (where $\sigma$-separation reduces to $d$-separation).



One important limitation of the $\sigma$-faithfulness assumption is that
it excludes the linear and discrete cases. In pioneering work \citet{RichardsonPhD1996} already 
proposed a constraint-based causal discovery algorithm (NL-CCD) that made use of the 
$\sigma$-separation Markov assumption, while assuming only the $d$-faithfulness
assumption (which is weaker than the $\sigma$-faithfulness assumption). 
In future work, we plan to investigate this setting as well, as well as
the possibility of extending our results to a setting that does not rule out selection bias.

\subsubsection*{Acknowledgements}

We are indebted to Jiji Zhang for contributing the proof of Proposition~\ref{prop:PAG=MEC}.
We thank the reviewers for their constructive feedback that helped us improve this paper.
This work was supported by the European Research Council (ERC) under the European Union's Horizon 2020 research and innovation programme (grant agreement 639466).

\newpage


\clearpage

\appendix

\section{PRELIMINARIES}\label{app:preliminaries}

In this section of the Supplementary Material, we briefly state all required definitions, notations and results from the literature to make the paper more self-contained.

\subsection{GRAPHS}

Here we briefly discuss various types of graphs (directed mixed graphs, maximal ancestral graphs, and partial ancestral graphs)
and their properties and relationships from the causal discovery literature. For more details, the reader may consult
the relevant literature \citep{SGS2000,RichardsonSpirtes02,Zhang2006,Zhang2008_AI,Zhang2008_JMLR}.

\subsubsection{DIRECTED MIXED GRAPHS (DMGs)}
A \emph{Directed Mixed Graph} (DMG) is a graph $\C{G} = \langle \C{V},\C{E},\C{F} \rangle$ with 
nodes $\C{V}$ and two types of edges: \emph{directed} edges $\C{E} \subseteq \{(i,j) : i, j \in \C{V}, i \ne j\}$, and 
\emph{bidirected} edges $\C{F} \subseteq \{\{i,j\} : i, j \in \C{V}, i \ne j\}$. 
We will denote a directed edge $(i,j) \in \C{E}$ as $i \rightarrow j$ or $j \ot i$, and
call $i$ a \emph{parent} of $j$. We denote all parents of
$j$ in the graph $\C{G}$ as $\pasub{\C{G}}{j} := \{i \in \C{V} : i \to j \in \C{E}\}$. We do not
allow for self-cycles $i \to i$ here, but multiple edges (at most one
of each type, i.e., at most three) between any pair of distinct nodes are allowed.
We will denote a bidirected edge $\{i,j\} \in \C{F}$ as $i \oto j$ or $j \oto i$.
Two nodes $i,j \in \C{V}$ are called \emph{adjacent in $\C{G}$} if 
if $i\to j \in \C{E}$ or $i \ot j \in \C{E}$ or $i \oto j \in \C{F}$.

A \emph{walk} between two nodes $i,j\in\C{V}$ is a tuple $\langle i_0,e_1,i_1,e_2,i_3,\dots,e_n,i_n \rangle$ 
of alternating nodes and edges in $\C{G}$ ($n \ge 0$), such that all $i_0,\dots,i_n \in \C{V}$,
all $e_1,\dots,e_n \in \C{E} \cup \C{F}$, starting with node $i_0=i$ and ending with node $i_n=j$,
and such that for all $k=1,\dots,n$, the edge $e_k$ connects the two nodes $i_{k-1}$ and
$i_k$ in $\C{G}$. If the walk contains each node at most once, it is called a \emph{path}.
A \emph{trivial walk (path)} consists just of a single node and zero edges.
A \emph{directed walk (path) from $i \in \C{V}$ to $j \in \C{V}$} is a walk (path) between $i$ and $j$ such that every edge 
$e_k$ on the walk (path) is of the form $i_{k-1} \to i_k$, i.e., every edge is directed and points away from $i$. 
By repeatedly taking parents, we obtain the \emph{ancestors} of $j$:
$\ansub{\C{G}}{j} := \{i \in \C{V} : i = i_0 \to i_1 \to \dots \to i_n = j \text{ in } \C{G}\}$.
Similarly, we define the \emph{descendants} of $i$: $\desub{\C{G}}{i} := \{j \in \C{V}: i = i_0 \to i_1 \to \dots \to i_n = j \text{ in } \C{G}\}$. In particular, each node is ancestor and descendant of itself.
A \emph{directed cycle} is a directed path from $i$ to $j$ such that in addition, $j \to i \in \C{E}$.
An \emph{almost directed cycle} is a directed path from $i$ to $j$ such that in addition, $j \oto i \in \C{F}$.
All nodes on directed cycles passing through $i \in \C{V}$ together form the \emph{strongly connected component}
$\sccsub{\C{G}}{i}:= \ansub{\C{G}}{i} \cap \desub{\C{G}}{i}$ of $i$.
We extend the definitions to sets $I \subseteq \C{V}$ by setting $\ansub{\C{G}}{I} := \cup_{i\in I} \ansub{\C{G}}{i}$, 
and similarly for $\desub{\C{G}}{I}$ and $\sccsub{\C{G}}{I}$.  

A directed mixed graph $\C{G}$ 
is \emph{acyclic} if it does not contain any directed cycle, in which case it is known as an 
\emph{Acyclic Directed Mixed Graph (ADMG)}. A directed mixed graph that does not contain 
bidirected edges is known as a \emph{Directed Graph (DG)}. If a directed mixed graph does not
contain bidirected edges and is acyclic, it is called a \emph{Directed Acyclic Graph (DAG)}.

A node $i_k$ on a walk (path) $\pi = \langle i_0,e_1,i_1,e_2,i_3,\dots,e_n,i_n \rangle$ in $\C{G}$ is said to 
form a \emph{collider on $\pi$} if it is a non-endpoint node ($1 \le k < n$) and the two edges $e_k,e_{k+1}$ 
meet head-to-head on their shared node $i_k$ (i.e., if the two
subsequent edges are of the form $i_{k-1} \to i_k \ot i_{k+1}$, $i_{k-1} \oto i_k \ot i_{k+1}$, 
$i_{k-1} \to i_k \oto i_{k+1}$, or $i_{k-1} \oto i_k \oto i_{k+1}$). Otherwise (that is, if it is an endpoint node, i.e., $k=0$ or $k=n$, or if the two subsequent edges are of the form $i_{k-1} \to i_k \to i_{k+1}$, $i_{k-1} \ot i_k \ot i_{k+1}$,
$i_{k-1} \ot i_k \to i_{k+1}$, $i_{k-1} \oto i_k \to i_{k+1}$, or $i_{k-1} \ot i_k \oto i_{k+1}$), 
$i_k$ is called a \emph{non-collider on $\pi$}.

The important notion of $d$-separation was first proposed by \citet{Pearl1986} in the context of DAGs:
\begin{definition}[$d$-separation]
We say that a walk $\langle i_0 \dots i_n \rangle$ in DMG $\C{G} = \langle \C{V},\C{E},\C{F} \rangle$ is \emph{$d$-blocked by $C \subseteq \C{V}$} if:
\begin{compactenum}[(i)]
  \item its first node $i_0 \in C$ or its last node $i_n \in C$, or
  \item it contains a collider $i_k \notin \ansub{\C{G}}{C}$, or
  \item it contains a non-collider $i_k \in C$.
\end{compactenum}
If all paths in $\C{G}$ between any node in set $A \subseteq \C{V}$ and any node in set $B \subseteq \C{V}$
are $d$-blocked by a set $C \subseteq \C{V}$, we say that $A$ is \emph{$d$-separated}
from $B$ by $C$, and we write $\dsep{A}{B}{C}{\C{G}}$.
\end{definition}

In the general cyclic case, the notion of $d$-separation is too strong, as was already pointed out by
\citet{Spirtes95}. A solution is to replace it with a non-trivial generalization of $d$-separation, 
known as $\sigma$-separation:
\begin{definition}[$\sigma$-separation \citep{ForreMooij_1710.08775}]
We say that a walk $\langle i_0 \dots i_n \rangle$ in DMG $\C{G} = \langle \C{V},\C{E},\C{F} \rangle$ is \emph{$\sigma$-blocked by $C \subseteq \C{V}$} if:
\begin{compactenum}[(i)]
\item its first node $i_0 \in C$ or its last node $i_n \in C$, or
\item it contains a collider $i_k \notin \ansub{\C{G}}{C}$, or
\item it contains a non-collider $i_k \in C$ that points to a 
neighboring node on the walk in another strongly connected component (i.e.,
$i_{k-1} \to i_k \to i_{k+1}$ or $i_{k-1}\oto i_k \to i_{k+1}$ with $i_{k+1} \notin \sccsub{\C{G}}{i_k}$,
$i_{k-1} \ot i_k \ot i_{k+1}$ or $i_{k-1}\ot i_k \oto i_{k+1}$ with $i_{k-1} \notin \sccsub{\C{G}}{i_k}$,
or $i_{k-1} \ot i_k \to i_{k+1}$ with $i_{k-1} \notin \sccsub{\C{G}}{i_k}$ or $i_{k+1} \notin \sccsub{\C{G}}{i_k}$).
\end{compactenum}
If all paths in $\C{G}$ between any node in set $A \subseteq \C{V}$ and any node in set $B \subseteq \C{V}$
are $\sigma$-blocked by a set $C \subseteq \C{V}$, we say that $A$ is \emph{$\sigma$-separated}
from $B$ by $C$, and we write $\sigmasep{A}{B}{C}{\C{G}}$.
\end{definition}

For a DMG $\C{G}$, define its \emph{$d$-independence model} to be
$$\IM_d(\C{G}) := \{ \langle A, B, C \rangle : A, B, C \subseteq \C{V}, \dsep{A}{B}{C}{\C{G}} \},$$
i.e., the set of all $d$-separations entailed by the graph, and its
\emph{$\sigma$-independence model} to be
$$\IM_\sigma(\C{G}) := \{ \langle A, B, C \rangle : A, B, C \subseteq \C{V}, \sigmasep{A}{B}{C}{\C{G}} \},$$
i.e., the set of all $\sigma$-separations entailed by the graph. 
For ADMGs, $\sigma$-separation is equivalent to $d$-separation, and hence, if $\C{G}$ is acyclic, then
$\IM_d(\C{G}) = \IM_\sigma(\C{G})$.
We call two DMGs $\C{G}_1$ and $\C{G}_2$ \emph{$\sigma$-Markov equivalent} if $\IM_\sigma(\C{G}_1) = \IM_\sigma(\C{G}_2)$,
and \emph{$d$-Markov equivalent} if $\IM_d(\C{G}_1) = \IM_d(\C{G}_2)$.

\subsubsection{DIRECTED MAXIMAL ANCESTRAL GRAPHS (DMAGs)}

The following graphical notion will be necessary for the definition of DMAGs.\footnote{We propose
an extension of this notion for DMGs in Definition~\ref{def:inducing_path}.}
\begin{definition}\label{def:inducing_path_acyclic}
Let $\C{G} = \langle \C{V}, \C{E}, \C{F} \rangle$ be an acyclic directed mixed graph (ADMG). 
An \emph{inducing path\footnote{This was called ``primitive inducing path'' in \citep{RichardsonSpirtes02}.} between two nodes $i,j \in \C{V}$} is a path in $\C{G}$ between $i$ and $j$ on which every node 
(except for the endnodes) is a collider on the path and an ancestor in $\C{G}$ of an endnode of the path.
\end{definition}
DMAGs can now be defined as follows \citep{Zhang2008_JMLR}:
\begin{definition}\label{def:DMAG}
  A directed mixed graph $\C{G} = \langle \C{V}, \C{E}, \C{F} \rangle$ is called a \emph{directed maximal ancestral graph (DMAG)} if all of the following conditions hold: 
  \begin{compactenum}
    \item Between any two different nodes there is at most one edge, and there are no self-cycles;
    \item The graph contains no directed or almost directed cycles (``ancestral'');
    \item There is no inducing path between any two non-adjacent nodes (``maximal'').
  \end{compactenum}
\end{definition}
With the following procedure from \citet{RichardsonSpirtes02}, one can define the DMAG induced by an ADMG:
\begin{definition}\label{def:DMAG_induced_by_ADMG}
  Let $\C{G} = \langle \C{V}, \C{E}, \C{F} \rangle$ be an ADMG. 
  The \emph{directed maximal ancestral graph induced by $\C{G}$} is denoted $\DMAG(\C{G})$ and is defined as
  $\DMAG(\C{G}) = \langle \tilde{\C{V}}, \tilde{\C{E}}, \tilde{\C{F}} \rangle$ such that $\tilde{\C{V}} = \C{V}$ and
  for each pair $u,v \in \C{V}$ with $u \ne v$, there is an edge in $\DMAG(\C{G})$ between $u$ and $v$ if and only if 
  there is an inducing path between $u$ and $v$ in $\C{G}$, and in that case the edge in $\DMAG(\C{G})$ connecting $u$ and $v$ is:
    $$\begin{cases}
      \text{$u \to  v$} & \text{if $u \in \ansub{\C{G}}{v}$}, \\
      \text{$u \ot  v$} & \text{if $v \in \ansub{\C{G}}{u}$}, \\
      \text{$u \oto v$} & \text{if $u \not\in \ansub{\C{G}}{v}$ and $v \not\in \ansub{\C{G}}{u}$.}
    \end{cases}$$
\end{definition}
This construction preserves the (non-)ancestral relations as well as the $d$-separations/connections.
We sometimes identify a DMAG $\C{H}$ with the set of ADMGs $\C{G}$ that induce $\C{H}$, i.e., such that $\DMAG(\C{G}) = \C{H}$.
For a DMAG $\C{H}$, we define its \emph{independence model} to be
$$\IM(\C{H}) := \{ \langle A, B, C \rangle : A, B, C \subseteq \C{V}, \dsep{A}{B}{C}{\C{H}} \},$$
i.e., the set of all $d$-separations entailed by the DMAG. We call two DMAGs
$\C{H}_1$ and $\C{H}_2$ \emph{Markov equivalent} if $\IM(\C{H}_1) = \IM(\C{H}_2)$. 

\subsubsection{DIRECTED PARTIAL ANCESTRAL GRAPHS (DPAGs)}

It is often convenient when performing causal reasoning to be able to represent a set of DMAGs in a compact way. 
For this purpose, \emph{partial ancestral graphs (PAGs)} have been introduced \citep{Zhang2006}.
Again, since we are assuming no selection bias for simplicity, we will only discuss \emph{directed} PAGs (DPAGs), that is, PAGs without undirected or circle-tail edges, i.e., edges of the form $\{\ttt,\ttc,\ctt\}$.
\begin{definition}
  We call a mixed graph $\C{G} = \langle\C{V},\C{E}\rangle$ with nodes $\C{V}$ and edges $\C{E}$ of the types $\{\to,\ot,\otc,\oto,\ctc,\cto\}$ a \emph{directed partial ancestral graph (DPAG)} if all of the following conditions hold:
  \begin{compactenum}
    \item Between any two different nodes there is at most one edge, and there are no self-cycles;
    \item The graph contains no directed or almost directed cycles (``ancestral'');
    \item There is no inducing path between any two non-adjacent nodes (``maximal'').
  \end{compactenum}
\end{definition}

We extend the definitions of (directed) walks, (directed) paths and colliders for directed mixed graphs to apply also to DPAGs. 
Edges of the form $i \ot j, i \otc j, i \oto j$ are called \emph{into $i$}, and similarly, edges of the form $i \to j, i \cto j, i \oto j$ are called \emph{into $j$}. Edges of the form $i \to j$ and $j \ot i$ are called \emph{out of $i$}.

Given a DMAG or DPAG, its induced \emph{skeleton} is an undirected graph with the same nodes and with an edge between any pair of nodes if and only if the two nodes are adjacent in the DMAG or DPAG.

One often identifies a DPAG with the set of all DMAGs that have the same skeleton as the DPAG, have an arrowhead (tail) on each edge mark for which the DPAG has an arrowhead (tail) at that corresponding edge mark, and for each circle in the DPAG, have either an arrowhead or a tail at the corresponding edge mark. 
We then say that the DPAG \emph{contains} these DMAGs.
Since each ADMG induces a unique DMAG, we can say that a DPAG contains an ADMG if and only if it contains the DMAG
induced by it.\footnote{In Definition~\ref{def:dpag_contains_dmg}, we propose to extend this to the notion that a DPAG contains a DMG.}


\subsection{FAST CAUSAL INFERENCE (FCI)}

When given as input an independence model $\IM(\C{H})$ of a DMAG $\C{H}$, FCI outputs a DPAG $\PAGFCI(\IM(\C{H}))$ that contains $\C{H}$ and is maximally informative \citep{Zhang2008_AI}, i.e., each edge mark that is identifiable from the independence model of $\C{H}$ is oriented as such in the DPAG. This DPAG is often referred to as the \emph{complete} DPAG for $\C{H}$.

The following key result seems to be generally known in the field, although we could not easily find a proof in the literature.\footnote{For DMAGs and POIPGs, this was already known for an earlier version of FCI; see Corollary 6.4.1 in \citet{SGS2000} and the proof in \citet{SpirtesVerma92}.}
The proof we provide here is due to Jiji Zhang [private communication].
\begin{proposition}\label{prop:PAG=MEC}
  Two MAGs $\C{H}_1,\C{H}_2$ are $d$-Markov equivalent if and only if $\PAGFCI(\IM_d(\C{H}_1)) = \PAGFCI(\IM_d(\C{H}_2))$.
\end{proposition}
\begin{proof}
We only give a proof for the ``if'' implication, the ``only if'' implication being obvious.
\citet{Ali++2009} showed that two MAGs are Markov equivalent if and only if they have the same skeletons and colliders with order.
This implies that colliders with order in any MAG are invariant in the corresponding Markov equivalence class and, by the soundness and arrowhead completeness of FCI, these will appear as colliders in the corresponding PAG. 
Consider two MAGs $\C{H}_1$ and $\C{H}_2$ that are not Markov equivalent. 
If they have different skeletons, their corresponding PAGs are not identical.
If they do have the same skeletons, then there must be at least one collider with order in $\C{H}_1$ that is not a collider with order in $\C{H}_2$, or vice versa.
By Lemma 3.13 in \citet{Ali++2009}, this would imply that there is at least one collider with order in $\C{H}_1$ that is not a collider in $\C{H}_2$, or vice versa.
Hence, because of the soundness and arrowhead completeness of FCI, their corresponding PAGs are not identical.
\end{proof}

Another important property of the FCI algorithm that we use is the following. 
It was stated in a different but equivalent formulation by \citet[Lemma 4.1, ][]{Ali++2005}.
\begin{lemma}[Lemma A.1 in \citep{Zhang2008_AI}]\label{lemm:Zhang_A1}
Let $\C{H}$ be a MAG and denote by $\C{P} = \PAGFCI(\IM(\C{H}))$ the corresponding complete PAG output by FCI.
Then for any three distinct vertices $a,b,c$:
if $a \sto b \cts c$ in $\C{P}$, then $a \sto c$ in $\C{P}$;
furthermore, if $a \to b \cts c$ in $\C{P}$, then $a \oto c$ is not in $\C{P}$.
\end{lemma}
Since this property is about arrowhead completeness, it already holds for the DPAG constructed in the first stage of the FCI 
algorithm, after running the arrowhead orientation rules $\C{R}$0--$\C{R}$4, but before running the tail orientation rules
$\C{R}$5--$\C{R}$10.

\subsection{STRUCTURAL CAUSAL MODELS (SCMs)}

In this subsection we state some of the basic definitions and results regarding Structural Causal Models.
Structural Causal Models (SCMs), also known as Structural Equation Models (SEMs), were
introduced a century ago by \citep{Wright1921} and popularized in AI by \citet{Pearl2009}.
We follow here the treatment in \citet{Bongers++_1611.06221v3} because it deals with cycles
in a rigorous way.
\begin{definition}\label{def:SCM}
A Structural Causal Model (SCM) is a tuple $\C{M} = \langle \C{I}, \C{J}, \BC{X}, \BC{E}, \B{f}, \Prb_{\BC{E}} \rangle$ of:
\begin{compactenum}[(i)]
\item a finite index set $\C{I}$ for the endogenous variables in the model;
\item a finite index set $\C{J}$ for the latent exogenous variables in the model (disjoint from $\C{I}$);
\item a product of standard measurable spaces $\BC{X} = \prod_{i \in \C{I}} \C{X}_i$, which define the
    domains of the endogenous variables; 
\item a product of standard measurable spaces $\BC{E} = \prod_{j \in \C{J}} \C{E}_j$, which define the
    domains of the exogenous variables;
\item a measurable function $\B{f} : \BC{X} \times \BC{E} \to \BC{X}$, the \emph{causal
    mechanism};
\item a product probability measure $\Prb_{\BC{E}} = \prod_{j \in \C{J}} \Prb_{\C{E}_j}$ on $\BC{E}$ specifying
    the \emph{exogenous distribution}.
\end{compactenum}
\end{definition}
Usually, the components of $\B{f}$ do not depend on all variables, which is formalized by:
\begin{definition}
Let $\C{M}$ be an SCM. We call $i \in \C{I} \cup \C{J}$ a parent of $k \in \C{I}$ if and only if there does not exist a measurable
function $\tilde f_k : \BC{X}_{\C{I}\setminus \{i\}} \times \BC{E}_{\C{J}\setminus\{i\}} \to \C{X}_k$ such that 
  for $\Prb_{\BC{E}}$-almost every $\B{e}$ and for all $\B{x} \in \BC{X}$, $x_k = \tilde f_k(\B{x}_{\C{I}\setminus \{i\}},\B{e}_{\C{J}\setminus \{i\}}) \iff x_k = f_k(\B{x},\B{e})$.
\end{definition}
This definition allows us to define the directed mixed graph associated to an SCM (which corresponds with the
DMG in Figure~\ref{fig:different_graphs}, our starting point for this work):
\begin{definition}
Let $\C{M}$ be an SCM. The induced \emph{graph} of $\C{M}$, denoted $\C{G}(\C{M})$, is defined as the directed mixed graph
with nodes $\C{I}$, directed edges $i_1 \to i_2$ iff $i_1$ is a parent of $i_2$, and bidirected edges
$i_1 \oto i_2$ iff there exists $j \in \C{J}$ such that $j$ is parent of both $i_1$ and $i_2$.
\end{definition}
If $\C{G}(\C{M})$ is acyclic, we call the SCM $\C{M}$ \emph{acyclic}, otherwise we call the SCM \emph{cyclic}. 
If $\C{G}(\C{M})$ contains no bidirected edges, we call the endogenous variables in the SCM $\C{M}$ \emph{causally sufficient}.

A pair of random variables $(\B{X},\B{E})$ is called a \emph{solution} of the SCM $\C{M}$ if
$\B{X} = (X_i)_{i \in \C{I}}$ with $X_i \in \C{X}_i$ for all $i \in \C{I}$,
$\B{E} = (E_j)_{j \in \C{J}}$ with $E_j \in \C{E}_j$ for all $j \in \C{J}$,
the distribution $\Prb(\B{E})$ is equal to the exogenous distribution $\Prb_{\BC{E}}$, and
the \emph{structural equations}:
$$X_i = f_i(\B{X}, \B{E})\quad\text{a.s.}$$
hold for all $i \in \C{I}$.

For acyclic SCMs, solutions exist and have a unique distribution that is determined by the SCM.
This is not generally the case in cyclic SCMs, as these could have no solution at all, or 
could have multiple solutions with different distributions.
\begin{definition}\label{def:unique_solvability_wrt}
An SCM $\C{M}$ is said to be \emph{uniquely solvable w.r.t.\ $\C{O} \subseteq \C{I}$} if there exists 
  a measurable mapping $\B{g}_{\C{O}} : \BC{X}_{(\pasub{\C{G}(\C{M})}{\C{O}}\setminus\C{O})\cap\C{I}} \times \BC{E}_{\pasub{\C{G}(\C{M})}{\C{O}} \cap \C{J}} \to \BC{X}_{\C{O}}$ 
such that for $\Prb_{\BC{E}}$-almost every $\B{e}$ for all $\B{x} \in \BC{X}$:
  \begin{equation*}\begin{split}
    &\B{x}_{\C{O}} = \B{g}_{\C{O}}(\B{x}_{(\pasub{\C{G}(\C{M})}{\C{O}}\setminus\C{O})\cap\C{I}}, \B{e}_{\pasub{\C{G}(\C{M})}{\C{O}}\cap\C{J}}) \\
    &\quad\iff\quad \B{x}_{\C{O}} = \B{f}_{\C{O}}(\B{x},\B{e}).
  \end{split}\end{equation*}
\end{definition}
Loosely speaking: the structural equations for $\C{O}$ have a unique solution for $\B{X}_{\C{O}}$ in terms of the other variables appearing in those equations.
If $\C{M}$ is uniquely solvable with respect to $\C{I}$ (in particular, this holds if $\C{M}$ is acyclic), then it induces a unique \emph{observational distribution 
$\Prb_{\C{M}}(\B{X})$}.

\subsubsection{SIMPLE STRUCTURAL CAUSAL MODELS}\label{sec:simple_scm}

In this work we restrict attention to a particular subclass of SCMs that has many convenient properties:
\begin{definition}\label{def:simple_scm}
An SCM $\C{M}$ is called \emph{simple} if it is uniquely solvable with respect to each subset $\C{O} \subseteq \C{I}$.
\end{definition}
All acyclic SCMs are simple. 
The class of simple SCMs can be thought of as a generalization of acyclic SCMs that allows for (sufficiently weak) cyclic causal relations, but preserves many of the convenient properties that acyclic SCMs have.
Simple SCMs provide a special case of the more general class of \emph{modular} SCMs \citep{ForreMooij_1710.08775}.
One of the key aspects of SCMs---which we do not discuss here in detail because we do not make use of it in this work---is their causal semantics, which is defined in terms of \emph{perfect interventions}.


Simple SCMs have the following convenient properties.
A simple SCM induces a unique observational distribution. The class of simple SCMs is
closed under marginalizations and perfect interventions.
Without loss of generality, one can assume that simple SCMs have no self-cycles. 
The graph of a simple SCM also has a straightforward causal interpretation:
\begin{definition}
  Let $\C{M}$ be a simple SCM. If $i \to j \in \C{G}(\C{M})$ we call $i$ a \emph{direct cause of $j$ according to $\C{M}$}.
  If there exists a directed path $i \to \dots \to j \in \C{G}(\C{M})$, i.e., if $i \in \ansub{\C{G}(\C{M})}{j}$, then we call
  $i$ a \emph{cause of $j$ according to $\C{M}$}. If there exists a bidirected edge $i \oto j \in \C{G}(\C{M})$, then we call
  $i$ and $j$ \emph{confounded according to $\C{M}$}.
\end{definition}
The same graph $\C{G}(\C{M})$ of a simple SCM $\C{M}$ also represents
the conditional independences that must hold in the observational distribution $\Prb_{\C{M}}(\B{X})$ of $\C{M}$.
\citet{ForreMooij_1710.08775} proved a Markov property for the general class of modular SCMs,
but we formulate it here only for the special case of simple SCMs:
\begin{theorem}[$\sigma$-Separation Markov Property]\label{thm:sigma_separation}
For any solution $(\B{X},\B{E})$ of a simple SCM $\C{M}$, and for all subsets $A,B,C \subseteq \C{I}$ 
of the endogenous variables:
$$\sigmasep{A}{B}{C}{\C{G}(\C{M})} \implies \indep{\B{X}_A}{\B{X}_B}{\B{X}_C}{\Prb_{\C{M}}(\B{X})}.$$
\end{theorem}
In certain cases, amongst which the acyclic case, the following stronger Markov property holds:
\begin{theorem}[$d$-Separation Markov Property]\label{thm:d_separation}
For any solution $(\B{X},\B{E})$ of an acyclic SCM $\C{M}$, and for all subsets $A,B,C \subseteq \C{I}$
of the endogenous variables:
$$\dsep{A}{B}{C}{\C{G}(\C{M})} \implies \indep{\B{X}_A}{\B{X}_B}{\B{X}_C}{\Prb_{\C{M}}(\B{X})}.$$
\end{theorem}

\subsubsection{FAITHFULNESS}

For a simple SCM $\C{M}$ with endogenous index set $\C{I}$ and observational distribution $\Prb_{\C{M}}(\B{X})$, 
we define its independence model to be
\begin{equation*}\begin{split}
  \IM(\C{M}) := \{ \langle A, B, C \rangle :{} &A, B, C \subseteq \C{I}, \\
  & \indep{\B{X}_A}{\B{X}_B}{\B{X}_C}{\Prb_{\C{M}}} \},
\end{split}\end{equation*}
i.e., the set of all (conditional) independences that hold in its (observational) distribution.

The typical starting point for constraint-based approaches to causal discovery from observational data is 
to assume that the data is modelled by an (unknown) SCM $\C{M}$, such that its observational
distribution $\Prb_{\C{M}}(\B{X})$ exists and satisfies a Markov property with respect to its graph
$\C{G}(\C{M})$. In other words, $\IM(\C{M}) \supseteq \IM_d(\C{G}(\C{M}))$ in the acyclic case,
and more generally, $\IM(\C{M}) \supseteq \IM_\sigma(\C{G}(\C{M}))$ for simple SCMs.

In addition, one usually assumes the \emph{faithfulness assumption} to hold \citep{SGS2000,Pearl2009}, 
i.e., that the graph explains \emph{all} conditional independences present in the 
observational distribution. We distinguish the $d$-faithfulness assumption:
$$\IM(\C{M}) \subseteq \IM_d(\C{G}(\C{M}))$$
and the $\sigma$-faithfulness assumption:\footnote{This notion is called \emph{``collapsed graph faithful''} in \citet{RichardsonPhD1996}.}
$$\IM(\C{M}) \subseteq \IM_\sigma(\C{G}(\C{M})).$$
Although for $d$-faithfulness there are some results that this assumption holds \emph{generically}
\citep{Meek1995} for certain parameterizations of acyclic SCMs, no such results are known for $\sigma$-faithfulness, 
although it has been conjectured \citep{Spirtes95} that at least weak completeness results can be shown.

\end{document}